\documentclass[11pt]{amsart}
\usepackage{amsmath,amsfonts,amsthm,mathrsfs,amssymb,amscd,comment,enumerate,amsxtra,url,tikz-cd}
\input{xypic}
\xyoption{all}

\usepackage{xcolor} 
\colorlet{mdtRed}{red!50!black}
\definecolor{dblue}{rgb}{0,0,.6}
\usepackage[colorlinks]{hyperref}
\hypersetup{linkcolor=blue,citecolor=dblue,filecolor=dullmagenta,urlcolor=mdtRed}

\numberwithin{equation}{section}
\newtheorem{theorem}[equation]{Theorem}
\newtheorem{corollary}[equation]{Corollary}
\newtheorem{lemma}[equation]{Lemma}
\newtheorem{proposition}[equation]{Proposition}
\newtheorem{definition}[equation]{Definition}

\newtheorem*{theorem*}{Theorem}
\newtheorem*{corollary*}{Corollary}
\newtheorem*{proposition*}{Proposition}

\theoremstyle{remark}

\def\subsection{
	\refstepcounter{equation}
	\noindent {\bf \arabic{section}.\arabic{equation}.}
}

\newcommand{\Z}{\mathbb{Z}}

\newcommand{\R}{\mathbb{R}}

\newcommand{\mf}[1]{\mathfrak{#1}}
\newcommand{\ms}[1]{\mathscr{#1}}
\newcommand{\mb}[1]{\mathbb{#1}}
\newcommand{\mc}[1]{\mathcal{#1}}
\renewcommand{\t}[1]{\tilde{#1}}

\usepackage{marginnote}
\usetikzlibrary{calc}

\begin{document}

\title[Nef Cones of some Quot Schemes]{Nef cones of some Quot schemes on a Smooth Projective Curve}

\author[C. Gangopadhyay]{Chandranandan Gangopadhyay} 

\address{Department of Mathematics, Indian Institute of Technology Bombay, Powai, Mumbai 400076, Maharashtra, India.} 

\email{chandra@math.iitb.ac.in} 

\author[R. Sebastian]{Ronnie Sebastian} 

\address{Department of Mathematics, Indian Institute of Technology Bombay, Powai, Mumbai 400076, Maharashtra, India.} 

\email{ronnie@math.iitb.ac.in} 



\maketitle

\begin{abstract}
 Let
$C$ be a smooth projective curve over $\mb C$. Let $n,d\geq 1$. Let 
$\mc Q$ be the Quot scheme parameterizing torsion quotients of the vector bundle $\mc O^n_C$ of degree $d$. In this article we study the nef cone of $\mc Q$. We give a complete description of the nef cone in the case of elliptic curves. We compute it in the case when $d=2$  and $C$ very general, in terms of the nef cone of the second symmetric product of $C$. In the case when $n\geq d$ and $C$ very general, we give upper and lower bounds for the Nef cone. In general, we give a necessary and sufficient criterion for a divisor on $\mc Q$ to be nef.
\end{abstract}

\section{Introduction}
Throughout this article we assume that 
the base field to be $\mb C$.
Let $X$ be a smooth projective variety and let $N^1(X)$ 
be the $\mb R$-vector space of $\mb R$-divisors modulo 
numerical equivalence. It is known that $N^1(X)$ is a 
finite dimensional vector space. The closed cone 
${\rm Nef}(X)\subset N^1(X)$ is the cone of all 
$\mb R$-divisors whose intersection product with 
any curve in $X$ is non-negative. 
It has been an interesting problem  to compute 
${\rm Nef}(X)$. For example, when $X=\mb P(E)$ where 
$E$ is a semistable vector bundle over a smooth 
projective curve, Miyaoka computed the 
${\rm Nef}(X)$ in \cite{Miyaoka}. In \cite{Biswas-Param}, 
${\rm Nef}(X)$ was computed in the case 
when $X$ is the Grassmann bundle  associated 
to a vector bundle $E$ on a smooth projective 
curve $C$, in terms of the Harder Narasimhan filtration of $E$. 
Let $C^{(d)}$ denote the $d$th symmetric product.
In \cite{Pa}, the author computed the 
${\rm Nef}(C^{(d)})$ in the case when $C$ is a 
very general curve of even genus and $d={\rm gon}(C)-1$. 
In \cite{Kouv-1993} ${\rm Nef}(C^{(2)})$ is computed in the 
case  when $C$ is very general and $g$ is a perfect square. 
In \cite{Ciro-Kouv} ${\rm Nef}(C^{(2)})$ was computed assuming 
the Nagata conjecture. We refer the reader to 
\cite[Section 1.5]{Laz} for more such examples and details.

The reader is referred to \cite{FGAex} for the definition and details
on Quot schemes. 
Let $E$ be a vector bundle over a smooth projective curve $C$. 
Fix a polynomial $P\in \mb Q[t]$. Let $\mc Q(E,P)$ denote 
the Quot scheme parametrizing quotients of $E$ with Hilbert 
polynomial $P$. 
In \cite{Str}, when $C=\mb P^1$, 
the quot scheme $\mc Q(\mc O^n_{C},P)$ is 
studied as a natural compactification of the set of all 
maps from $C$ to some Grassmannians of a fixed degree. 
In this article we will consider the case when $P=d$ a constant, 
that is, when $\mc Q(E,d)$ parametrizes torsion quotients 
of $E$ of degree $d$. For notational convenience, we will 
denote $\mc Q(E,d)$ by 
$\mc Q$, when there is no possibility of confusion. 
It is known that $\mc Q$ is a smooth projective 
variety. Many properties of $\mc Q$ have been studied. 
In \cite{Bis-Dh-Hu} the Brauer group of 
$\mc Q(\mc O^n_C,d)$ is computed. 
In \cite{BGL}, the Betti cohomologies of $\mc Q(\mc O^n_C,d)$ 
are computed, $\mc Q(O^n_C,d)$ 
has been intepreted as the space of higher rank divisors of 
rank $n$, and an analogue of the Abel-Jacobi map was constructed. 
In \cite{G19}, the automorphism group scheme 
of $\mc Q(E,d)$ was computed in the case when either 
${\rm rk}~E\geq 3$ or $E$ is semistable and genus of $C$ 
satisfies $g(C)>1$.  
In  \cite{GS}, the $S$-fundamental group scheme of $\mc Q(E,d)$ was 
computed. 

In this article, we address the question of computing ${\rm Nef}(\mc Q)$. 
Recall that we have a map $\Phi:\mc Q\to C^{(d)}$ (a precise
definition can be found, for example, in \cite{GS}). 
For notational convenience, for a divisor $D\in N^1(C^{(d)})$ we will denote 
its pullback $\Phi^*D\in N^1(\mc Q)$ by
$D$, when there is no possibility of confusion.
The line bundle $\mc O_{\mc Q}(1)$ is defined in Definition \ref{def-O(1)}.
In Section \ref{section-nef-C^{(d)}} we recall the results
we need on ${\rm Nef}(C^{(d)})$. 
In Section \ref{section on picard group and neron-severi group of quot schemes} we compute 
${\rm Pic}(\mc Q)$. 
\begin{theorem*}[Theorem \ref{thm1}]
	${\rm Pic}(\mathcal{Q})=\Phi^*{\rm Pic}(C^{(d)})\bigoplus \Z[\mathcal{O}_{\mathcal{Q}}(1)] $\,.
\end{theorem*}
As a corollary (Corollary \ref{neron-severi of quot scheme}) we get that
$N^1(\mc Q)\cong N^1(C^{(d)})\oplus \mb R[\mc O_{\mc Q}(1)]$.
The computation of $N^1(\mc Q)$ can also be found in \cite{Bis-Dh-Hu}.
As a result, when $C\cong \mb P^1$, since $C^{(d)}\cong \mb P^d$, 
we have that the $N^1(\mc Q)$ is $2$-dimensional and we 
prove that its nef cone is given as follows.
\begin{theorem*}[Theorem \ref{nef cone of quot schemes over projective line}]
	Let $C=\mb P^1$. Let $E=\bigoplus\limits^k_{i=1} \mc O(a_i)$ with $a_i\leq a_j$ for $i<j$. 
	Let $d\geq 1$.
	Then 
\begin{align*}	
	{\rm Nef}(\mc Q(E,d)) = &  \mb R_{\geq 0}([\mc O_{\mc Q(E,d)}(1)]+(-a_1+d-1)[\mc O_{\mb P^d}(1)])+\mb R_{\geq 0}[\mc O_{\mb P^d}(1)]\,.
\end{align*}
\end{theorem*}
\noindent
Note that this theorem was already known in the case
when $E=V\otimes \mc O_{\mb P^1}$, for a vector 
space $V$ over $k$ (\cite[Theorem 6.2]{Str}). 

For the rest of the introduction, we will assume 
$E=V \otimes \mc O_C$ with ${\rm dim}_kV=n$ and denote 
by $\mc Q=\mc Q(n,d)$ the Quot scheme $\mc Q(E,d)$. Let 
us consider the case $g=1$. In this case, $N^1(\mc Q)$ is 
three-dimensional (see Proposition \ref{dim-N^1}), 
and we prove that its nef cone is given as follows
(see Definition
\ref{three divisors of symmetric product of curves} for notations).

\begin{theorem*}[Theorem \ref{O(1)+Delta/2 is nef for elliptic curves}]
Let $g=1$, $n\geq 1$ and $\mc Q=\mc Q(n,d)$.  
	Then the class $[\mc O_{\mc Q}(1)]+[\Delta_d/2]\in N^1(\mc Q)$ is nef.
	Moreover, 
	$${\rm Nef}(\mc Q)= \mb R_{\geq 0}([\mc O_{\mc Q}(1)]+[\Delta_d/2])+
	\mb R_{\geq 0}[\theta_d]+\mb R_{\geq 0}[\Delta_d/2]\,.$$
\end{theorem*}

From now on assume that $g\geq 2$ and $C$ is very general.
See Definition \ref{def-alpha_t} for the definition of 
$t$ and $\alpha_t$. When $d=2$ we have the following result. 
\begin{theorem*}[Theorem \ref{cone-d=2}]
	Let $g\geq 2$ and $C$ be very general. Let $d=2$. Consider the Quot scheme $\mc Q=\mc Q(n,2)$. 
	Then $${\rm Nef}(\mc Q)=\mb R_{\geq 0}([\mc O_{\mc Q}(1)]+\dfrac{t+1}{g+t}[L_0])+
	\mb R_{\geq 0}[L_0]+\mb R_{\geq 0}[\alpha_t]\,.$$ 
\end{theorem*}
Precise values of $t$ are known for small genus. When $g\geq 9$
it is conjectured that $t=\sqrt{g}$. This is known when 
$g$ is a perfect square. The precise statements have been mentioned 
after Theorem \ref{cone-d=2}.

In general (without any assumptions on $n$ and $d$), 
we give a criterion for certain line bundle on $\mc Q$ to be nef
in terms of its pullback along certain natural maps from 
products $\prod_i C^{(d_i)}$, see subsection \ref{subsection-criterion-for-nefness}
for notation. 

\begin{theorem*}[Theorem \ref{criterion-for-nefness}]
	Let $\beta\in N^1(C^{(d)})$.
	Then the class $[\mc O_{\mc Q}(1)]+\beta \in N^1(\mc Q)$ is nef iff the class 
	$[\mc O(-\Delta_{\mathbf{d}}/2)]+\pi^*_{\mathbf{d}}\beta\in N^1(C^{(\mathbf{d})})$ 
	is nef for all $\mathbf{d}\in \mc P^{\leq n}_d$.
\end{theorem*}
Using the above we show that certain classes are in ${\rm Nef}(\mc Q)$.
Define
\begin{equation}\label{lower bound along the theta side}
	\kappa_1:=[\mc O_{\mc Q}(1)]+\mu_0[L_0]+\dfrac{d+g-2}{dg}[\theta_d]
	\qquad \kappa_2=[\mc O_{\mc Q}(1)]+\dfrac{g+1}{2g}[L_0] \in N^1(\mc Q)\,.
\end{equation}	


\begin{proposition*}[Proposition \ref{another lower bound for NEF cone}]
	Let $g\geq 1$, $n\geq 1$ and $\mc Q=\mc Q(n,d)$.  
	Then  
	$${\rm Nef}(\mc Q)\supset \mb R_{\geq 0}\kappa_1+
	\mb R_{\geq 0}\kappa_2+\mb R_{\geq 0}[\theta_d]+\mb R_{\geq 0}[L_0]\,.$$
\end{proposition*}

Now consider the case when $n\geq d\geq {\rm gon}(C)$. 
Then ${\rm Nef}(C^{(d)})$ is generated by $\theta_d$ and $L_0$ 
(see Definitions \ref{def-L_0} and 
\ref{three divisors of symmetric product of curves}). In this case 
we give the following upper bound for the nef 
cone in Proposition \ref{nef cone of quot schemes }.
Let $\mu_0:=\dfrac{d+g-1}{dg}$. Then 
$${\rm Nef}(\mc Q)\subset \mb R_{\geq 0}([\mc O_{\mc Q}(1)]+\mu_0[L_0])+\mb R_{\geq 0}[\theta_d]+\mb R_{\geq 0}[L_0]\,.$$
When $d\geq {\rm gon}(C)$, in Lemma \ref{B_L-not-ample} we show that 
any convex linear combination of the $\kappa_1$ and 
$\theta_d$ is nef but not ample. In particular, 
any such class lies on the boundary of ${\rm Nef}(\mc Q)$.
Similarly, in Corollary \ref{O(1)+mu^2L_0-not-ample} we show
when $n\geq d$, any convex linear combination of the 
class $\kappa_2$
and $L^{(d)}_0$ is nef but not ample. So any such class 
lies on the boundary of ${\rm Nef}(\mc Q)$. 
\begin{equation}\label{picture-nef-cone}
	\begin{tikzpicture}[scale = .7]
	\begin{scope}[thick,font=\scriptsize]
	\draw [dashed] (-3,0) -- (.8,0);
	\draw (.8,0) -- (4,0);
	\draw (-3,0) -- (1,-3);
	\filldraw (1.6,2.4) circle[radius=0.04cm]; 
	\path [draw=none,fill={rgb:black,1;white,2},semitransparent] (-3,0)--(.47,3.55)--(1,-3)--(-3,0);
	\draw (-3,0) -- (.47,3.55);
	\path [draw=none,fill={rgb:black,1;white,4},semitransparent] (-3,0)--(.47,3.55)--(4,0)--(-3,0);
	\filldraw (.565,2.3) circle[radius=0.04cm]; 
	\draw (-3,0) -- (.565,2.3); 
	\draw [dashed] (-3,0) -- (1.6,2.4); 
	\filldraw (.47,3.55) circle[radius=0.04cm];
	\draw [dashed] (1.6,2.4) -- (.565,2.3); 
	\draw (1,-3) -- (.47,3.55);
	\draw (4,0) -- ($(4,0)!5cm!(2.5,1.5)$);
	\draw (1,-3) -- (4,0);
	\filldraw (4,0) circle[radius=0.04cm];
	\node at (4,0) [right]{$(\textbf{C})$}; 
	\filldraw (1,-3) circle[radius=0.04cm];
	\node at (1,-3) [below]{$(\textbf{B})$}; 
	\node at (.3,2.2) [left]{$(\textbf{D})$}; 
	\node at (1.8,2.5) [right]{$(\textbf{E})$}; 
	\node at (.47-.3,3.55) [left]{$(\textbf{A})$}; 
	\node at (-3,0) [left]{$(\textbf{O})$};
	\end{scope}
	\end{tikzpicture}
\end{equation}

\begin{enumerate}[(1)]
	\item $(\textbf{A})=[\mc O_{\mc Q}(1)]+\mu_0[L_0]$
	\item $(\textbf{B})=[\theta_d]$
	\item $(\textbf{C})=[L_0]$
	\item $(\textbf{D})=\tau \kappa_1=\tau([\mc O_{\mc Q}(1)]/2+\mu_0[L_0])+(1-\tau)[\theta_d]\qquad
			\tau=\dfrac{1}{1+\frac{d+g-2}{dg}}$  
	\item $(\textbf{E})=\rho\kappa_2=\rho([\mc O_{\mc Q}(1)]+\mu_0[L_0]) + (1-\rho)[L_0]\qquad
		\rho=\dfrac{1}{1+\frac{g+1}{2g}-\frac{d+g-2}{dg}}$
\end{enumerate}
In terms of the above diagram, we have that when $n\geq d\geq {\rm gon}(C)$
$$\langle \overline{OD}, \overline{OE}, \overline{OC}, \overline{OB}\rangle 
	\subset {\rm Nef}(\mc Q)\subset \langle \overline{OA}, \overline{OC}, \overline{OB} \rangle \,. $$

We do not know if the inclusion in the right is 
an equality when $n\geq d\geq {\rm gon}(C)$. 
This is same as saying that $[\mc O_{\mc Q}(1)]+\mu_0[L_0]$
is nef when $n\geq d\geq {\rm gon}(C)$.
In Section \ref{curves-over-small-diagonal} we give a 
sufficient condition for when the pullback 
of $[\mc O_{\mc Q}(1)]+\mu_0[L_0]$ along a map $D\to \mc Q$ 
is nef. However, when $d=3$ we have the following result.

\begin{theorem*}[Theorem \ref{cone-d=3}]
	Let $C$ be a very general curve of genus $2\leq g(C)\leq 4$. 
	Let $n\geq3$ and let $\mc Q=\mc Q(n,3)$. Let $\mu_0=\dfrac{g+2}{3g}$
	Then 
	$${\rm Nef}(\mc Q)= \mb R_{\geq 0}([\mc O_{\mc Q}(1)]+\mu_0[L_0])+
	\mb R_{\geq 0}[\theta_d]+\mb R_{\geq 0}[L_0]\,.$$
\end{theorem*}

Some of the results above can be improved in the case
when $g=2k$ using the 
results in \cite{Pa}. (See Proposition \ref{nef-B_L-Pa}.)

\section{Nef cone of $C^{(d)}$}\label{section-nef-C^{(d)}}

We follow \cite[\S 2]{Pa} for this section.
Assume that either  $C$ is an elliptic curve or is a very general curve of genus $g\geq 2$. 
Then it is known that the Neron-Severi space is  $2$-dimensional. 
So in this case, to compute the nef cone, 
it is enough to give two classes in $N^1(C)$ which are nef but not ample. 

For any smooth projective curve and $d\geq 2$ 
(not just a very general curve) 
there is a natural line bundle $L_0$ on $C^{(d)}$ which is nef but not ample. 
This line bundle is constructed in the following manner. Consider the map 
$$\phi:C^d \to J(C)^{d \choose 2}\,,$$
$$(x_i)\mapsto (x_i-x_j)_{i<j}\,.$$
Let $p_{ij}$ denote the projections from $J(C)^{d \choose 2}$. 
Since $\phi$ is not finite, as it contracts the diagonal,
the line bundle $\phi^*(\otimes p^*_{ij}\Theta)$ is nef but not ample. 
This line bundle is invariant under the action of $S_d$ on $C^d$. 
This follows from the fact that 
$\Theta$ in $J(C)$ is invariant under the involution $L\mapsto L^{-1}$. 
\begin{definition}\label{def-L_0}
	$\phi^*(\otimes p^*_{ij}\Theta)$ descends to a line bundle $L_0$ on $C^{(d)}$. 
\end{definition}
Since $\phi$ contracts the small diagonal 
$\delta:C\hookrightarrow C^{(d)}$, we have 
$\delta^*[L_0]=0$. Hence $L_0$ is nef but not ample \cite[Lemma 2.2]{Pa}.
Therefore, in the case when $C$ is very general, 
computing the nef cone of $C^{(d)}$ boils down 
to finding another class which is nef but not ample. 

In the case when $d\geq {\rm gon}(C)=:e$, \cite[Lemma 2.3]{Pa}
we can easily construct another line bundle which is nef but not ample: 
Then we have a map $g_e:C\to \mb P^1$ of degree $e$. 
This induces a closed immersion 
$\mb P^1\to C^{(e)}$ with  $v\mapsto [(g_e)^{-1}(v)]\in C^{(e)}$. 
This in turn gives a closed immersion 
$\mb P^1\to C^{(d)}$ with $v\mapsto [(g_e)^{-1}(v)+(d-e)x]$ for some point $x\in C$. 
\begin{definition}\label{p1-in-C^(d)}
	Denote the class of this $\mb P^1$ in $N_1(C^{(d)})$ by $[l']$. 
\end{definition}
The composition $\mb P^1\to C^{(d)}\xrightarrow{u_d} J(C)$ is constant, 
since there can be no non-constant maps from $\mb P^1\to J(C)$. 
Hence $u_d:C^{(d)}\to J(C)$ is not finite and we get that 
$u_d^*\Theta$ is nef but not ample. 
\begin{definition}\label{def-theta_d}
	Define $\theta_d:=u_d^*\Theta$.
\end{definition}

\noindent 
Recall that over $C^{(d)}$ we have natural divisors \cite[\S 2]{Pa}:

\begin{definition}\label{three divisors of symmetric product of curves}
	Define
\begin{enumerate}
	\item $\theta_d$
	\item the big diagonal $\Delta_d\hookrightarrow C^{(d)}$
	\item If $i_{d-1}:C^{(d-1)}\to C^{(d)}$
	is the map given by $D\mapsto D+x$ for a point $x\in C$,
	then the image $i_{d-1}(C^{(d-1)})$. This divisor will be denoted $[x]$. 
\end{enumerate}
\end{definition}

It is known that when $g=1$ or $C$ is very general of $g\geq 2$, 
then $N^1(C^{(d)})$ is of dimension $2$ and any two of the 
above three forms a basis. 

By abuse of notation, let us denote the class ($\delta$ is the small diagonal) 
$[\delta_*(C)]\in N_1(C^{(d)})$
by $\delta$. We summarise the above discussion in the following theorem.

\begin{proposition}\cite[Proposition 2.4]{Pa}\label{nef-cone-C^(d)}
	When $d\geq {\rm gon}(C)$, we have:
	\begin{enumerate} 
		\item ${\rm Nef}(C^{(d)})=\mb R_{\geq 0}[L_0]\oplus \mb R_{\geq 0}[\theta_d]$\,,
		\item $\overline{NE}(C^{(d)})=\mb R_{\geq 0}[l']\oplus \mb R_{\geq 0}[\delta]$\,.
	\end{enumerate}
	The above basis are dual to each other.
\end{proposition}

We will need to write $[L_0]$ in terms of $[x]$ and $[\theta_d]$, for which we need 
the following computations.
Define $$\delta':C \xrightarrow{f} C^d \to C^{(d)}$$
where the first map is given by $x\mapsto (x,x_1,\dots,x_{d-1})$.

\begin{lemma}
	Let $d\geq 1$. We have the following 
	\begin{enumerate}
		\item ${\rm deg} (\delta^*[\theta_d])=d^2g$
		\item ${\rm deg} (\delta'^*[\theta_d])=g$
		\item ${\rm deg} (\delta^*[x])=d$
		\item ${\rm deg} (\delta'^*[x])=1$
	\end{enumerate}
\end{lemma}
\begin{proof}
	Recall that $\theta_d=u_d^*\Theta$, where
	$u_d:C^{(d)} \to  J(C) $ is given by
	$D    \mapsto    \mc O(D-dx_0)$
	for a fixed point $x_0\in C$. Therefore the composition 
	$u_d\circ \delta:C\to J(C)$ is given by 
	$x \mapsto dx   \mapsto \mc O(dx-dx_0)$,
	which is the map 
	$$C \xrightarrow{u_1} J(C) \xrightarrow{\times d} J(C) \,.$$
	The pullback of $\Theta$ under the map 
	$J(C)\xrightarrow{\times d} J(C)$ is $\Theta^{d^2}$ and 
	the degree of the pullback of $\Theta$ under the map 
	$u_1:C\to J(C)$ is $g$. Hence degree of $\delta^*\theta_d=d^2g$. 
	This proves (1).
	
	The composition $u_d\circ \delta':C \to J(C)$ is given by 
	$C \to  C^{(d)} \to     J(C)$
	\begin{align*}
	x \mapsto & x+\sum_{i=1}^{d-1} x_i   \mapsto \mc O(x+\sum_{i=1}^{d-1} x_i-dx_0)
	\end{align*}
	which is the composition
	$C\xrightarrow{u_1} J(C)\xrightarrow{t_a} J(C)$, 
	where $t_a$ is translation by an element in $J(C)$.
	Hence degree of $\delta'^*\theta_d=g$. This proves (2).
	
	For a line bundle $L$ on $C$, we will denote by $L^{\boxtimes d}$
	to be the unique line bundle on $C^{(d)}$, whose pullback 
	under the quotient map $\pi:C^d\to C^{(d)}$ is $\bigotimes_{i=1}^d p_i^*L$.
	Recall that by \cite[\S2]{Pa}, we have 
	that $[x]=[\mc O(x)^{\boxtimes d}]$ for a point $x\in C$. 
	By definition under the map $\pi:C^d\to C^{(d)}$ the 
	pullback of $\mc O(x)^{\boxtimes d}$ is $\bigotimes^d_{i=1}p^*_i\mc O(x)$. 
	Now $\delta:C\hookrightarrow C^{(d)}$ is the composition
	$C \to  C^d \rightarrow C^{(d)} $
	$$x \mapsto  (x,..,x) \mapsto dx \,.$$
	Hence we get that the pullback of $\mc O(x)^{\boxtimes d}$ to $\delta$ is $\mc O(dx)$. 
	Therefore degree of $\delta^*[x]=d$. This proves (3).
	
	We know $\delta'$ is the composition
	$C \to  C^d \rightarrow C^{(d)} $
	$$x \mapsto  (x,x_1,..,x_{d-1}) \mapsto x+x_1+\dots+x_{d-1} \,.$$
	Hence we get that $\delta'^*[x]=\mc O(x)$.
	Therefore degree of $\delta'^*[x]=1$. This proves (4).
\end{proof}

\begin{lemma}\label{E in terms of various bases}
	Let $g,d\geq 1$. Let $\mu_0:=\dfrac{d+g-1}{dg}$. Then 
	\begin{align*}
	[L_0]&=dg [x]-[\theta_d]\\
	&=(dg-d-g+1).[x]+[\Delta_d/2]\\
	&=(\dfrac{1}{\mu_0}-1)[\theta_d]+\dfrac{1}{\mu_0}[\Delta_d/2]\,.
	\end{align*}
\end{lemma}
\begin{proof}
	Let $[L_0]=a[\theta_d]+b[x]$. We need two equations to solve for $a$ and $b$.
	The first equation is $\delta^*[L_0]=0$. Recall 
	$$\delta':C \xrightarrow{f} C^d \to C^{(d)}$$
	where the first map is given by $x\mapsto (x,x_1,\dots,x_d)$.
	Hence 
	$$\delta'^*[L_0]=f^*\phi^*(\otimes p^*_i\Theta)\,.$$
	Now the composition
	$$C\xrightarrow{f} C^d\xrightarrow{\phi} J(C)^{d \choose 2}$$
	is given by $x\mapsto (x-x_1,x-x_2,\dots,x-x_{d-1},x_i-x_j)_{i<j}$.
	Hence 
	$${\rm deg}(\delta'^*[L_0])=\sum\limits_{i=1}^{d-1} {\rm deg}(\theta_1)=(d-1)g\,.$$
	This will be our second equation.
	
	We use these two equations and the preceding computations to 
	compute $a$ and $b$.
	\begin{align*}
	0&={\rm deg}(\delta^*[L_0])\\
	&= a.{\rm deg}(\delta^*[\theta_d])+b.{\rm deg}(\delta^*[x]) \\
	&= ad^2g+bd\,.
	\end{align*}
	Therefore 
	$$b=-adg\,.$$
	Now using the second equation we get
	\begin{align*}
	(d-1)g&={\rm deg}(\delta'^*[L_0])\\
	&=  a.{\rm deg}(\delta'^*[\theta_d])+b.{\rm deg}(\delta'^*[x])\\
	&= ag+b                      \\
	&= ag-adg   =  ag(1-d)\,.
	\end{align*}
	Therefore 
	$$a=-1,\qquad  b=dg\,.$$
	Hence we get $[L_0]=dg[x]-[\theta_d]$. For the other 
	two equalities, we use the relation 
	$$[\theta_d]=(d+g-1)[x]-[\Delta_d/2]$$
	between $[x],[\Delta_d/2]$ and $[\theta_d]$ \cite[Lemma 2.1]{Pa}.
\end{proof}

\section{Picard group and Neron-Severi group of $\mc Q$}\label{section on picard group and neron-severi group of quot schemes}
Let $E$ be a locally free sheaf over $C$. Throughout this section 
$\mc Q$ will denote the Quot scheme $\mc Q(E,d)$ 
which parametrizes torsion quotients of $E$ of degree d.
In this section we compute the Picard group of $\mc Q$, and the 
vector spaces $N^1(\mc Q)$ and $N_1(\mc Q)$.

\begin{lemma}\label{pullback-lemma}
	Let $S$ be a scheme over $k$. Let $F$ be a 
	coherent sheaf over $C\times S$ which is $S$-flat 
	and for all $s\in S$, $F|_{C\times s}$ is a torsion 
	sheaf over $C$ of degree $d$. Let 
	$p_S:C\times S\to S$ be the projection. Then 
	\begin{enumerate}[(i)]
		\item $p_{S*}(F)$ is locally free of rank $d$ 
		and $\forall s\in S$ the natural map 
		$p_{S*}(F)|_s\to H^0(C,F|_{C\times s})$ is an isomorphism. 
		\item Assume that we are given a morphism $\phi:T\to S$. We have the following diagram:
		\[
		\begin{tikzcd}
		C\times T \ar[r,"id\times \phi"] \ar[d,"p_T"] & C\times S \ar[d,"p_S"]\\
		T \ar[r,"\phi"] & S
		\end{tikzcd}
		\]
		Then the natural morphism
		$$\phi^*p_{S*}(F)\to (p_T)_*(id\times \phi)^*F$$ is an isomorphism.
	\end{enumerate}
\end{lemma}
\begin{proof}
	Since $F|_{C\times s}$ is a 
	torsion sheaf for all $s\in S$, we have $H^1(C,F|_{C\times s})=0$. 
	By \cite[Chapter III, Theorem 12.11(a)]{Ha}  we get 
	$R^1p_{S*}(F)=0$. Using \cite[Chapter III, Theorem 12.11(b)]{Ha}
	(ii) with $i=1$ we get that the morphism 
	$p_{S*}(F)|_s\to H^0(C,F|_{C\times s})$ is surjective. 
	Again using the same  
	with $i=0$ we get that $p_{S*}(F)$ is locally free of rank $d$
	and the map $p_{S*}(F)|_s\to H^0(C,F|_{C\times s})$ is an isomorphism.
	
	Since $F$ is $S$-flat it follows that $(id\times \phi)^*F$ is $T$-flat. 
	Applying the above we see $\phi^*p_{S*}(F)$ and $(p_T)_*(id\times \phi)^*F$ 
	are locally free of rank $d$. For 
	each $t\in T$ we have the commutative diagram:
	\[
	\begin{tikzcd}
	\phi^*p_{S*}(F)|_t=p_{S*}(F)|_{\phi(t)} \ar[r] \ar[d] & (p_T)_*(id\times \phi)^*F|_t \ar[d] & \\
	H^0(C,F|_{C\times \phi(t)}) \ar[r,equal] & H^0(C,(id\times \phi^*)F|_{C\times t}) 
	\end{tikzcd}
	\]
	By the first part we get that the vertical arrows 
	are isomorphisms. Hence we get that the first row 
	of the diagram is an isomorphism. Therefore
	$$\phi^*p_{S*}(F)\to (p_T)_*(id\times \phi)^*F$$ 
	is a surjective morphism of vector bundles of same rank and hence an isomorphism.
\end{proof}

We define a line bundle on $\mathcal{Q}$.
Let us denote the projections $C\times \mathcal{Q}$ 
to $C$ and $\mathcal{Q}$ by $p_{C}$ and $p_{Q}$ respectively. 
Then we have the universal quotient 
$p^*_CE \rightarrow \mathcal{B}_{\mc Q}$ over $C\times \mc Q$. By Lemma 
\ref{pullback-lemma}, 
$p_{\mathcal{Q}*} (\mathcal{B}_{\mc Q}) $ is a vector bundle of rank 
$d$. 
\begin{definition}\label{def-O(1)}
	Denote the line bundle 
{\rm det($p_{\mc {Q}*} (\mathcal{B}_{\mc Q})$)} by 
$\mc {O}_{\mc {Q}}(1)$.
\end{definition}

Denote the $d$-th symmetric product of $C$ by $C^{(d)}$.
Recall the Hilbert-Chow map 
$\Phi: \mathcal{Q}\rightarrow C^{(d)}$ 
which sends $[E\rightarrow B]$ to $\sum l(B_{p})p$, 
where $l(B_{p})$ is the length of the $\mathcal{O}_{C,p}$-module $B_{p}$.  
Therefore, we have the pullback 
$\Phi^*:{\rm Pic}(C^{(d)})\to {\rm Pic}(\mathcal{Q})$ 
which is in fact an inclusion. 
To see this, recall that the fibres of
$\Phi$ are projective integral varieties 
\cite[Corollary 6.6]{GS} and $\Phi$ is flat \cite[Corollary 6.3]{GS}. 
Hence $\Phi_*(\mathcal O_{\mc{Q}})=\mc{O}_{C^{(d)}}$.
Now by projection formula
$\Phi_{*}\Phi^*L\cong L$ for all $L\in {\rm Pic}(C^{(d)})$
and the statement follows.

The big diagonal is the image of the map $C\times C^{(d-2)}\to C^{(d)}$ 
given by $(x,A)\mapsto 2x +A$. 
Let us denote the big diagonal in $C^{(d)}$ by 
$\Delta$. Let $U_{C}:=C^{(d)}\setminus \Delta$ and 
$\mathcal{U}:=\Phi^{-1}(U_C)$. Then $\mc U\subset \mc Q$.
\begin{lemma}\label{generically the line bundle is O(n)}
	For any line bundle $\mc L\in {\rm Pic}(\mc Q)$, $\exists$ 
	an unique $n\in \mb Z$ such that 
	$(\mc L\otimes \mc O_{\mc Q}(-n))|_{\Phi^{-1}(p)}\cong \mc O_{\Phi^{-1}(p)}$ 
	for all $p\in U_C$. 
\end{lemma}

\begin{proof}
	Let $\pi:\mb P(E) \to C$ 
	be the projective bundle associated to $E$ 
	and let $\mc O_{\mb P(E)}(1)$ be the universal line bundle over $\mb P(E)$. 
	Let $Z=\mb P(E)^d$. Let $p_i:Z\to \mb P(E)$ be the $i$-th projection. 
	Let $\pi_{d}:Z \to C^d$ be the product map.
	The symmetric group $S_d$ acts on $Z$ and the map 
	$\pi_d$ is equivariant for this action.
	Let $\psi:C^d\to C^{(d)}$ be the quotient map. 
	Define $U_{Z}:=(\psi\circ \pi_d)^{-1}(U)$. 
	
	Let $c\in C$ be a closed point and let $k_{c}$ denote 
	the skyscraper sheaf supported at $c$. A closed 
	point of $\mb P(E)$ which maps to $c\in C$ 
	corresponds to a quotient $E\to E_c\to k_c$.
	Recall that we have a map \cite[Theorem 2.2(a)]{G19}
	$$\tilde{\psi}:U_Z\to \mc U$$
	which sends a closed point 
	$$(E_{c_i}\to k_{c_i})^d_{i=1}\in U_Z$$ 
	to the quotient 
	$$E\to \bigoplus\limits_i E_{c_i}\to \bigoplus\limits_i k_{c_i}\in \mc U\,.$$
	So we have a commutative diagram:
	\[
	\begin{tikzcd}
	U_Z            \ar[r,"\tilde{\psi}"] \ar[d,"\pi_d"] &  \mc U \ar[d,"\Phi"] \\
	\psi^{-1}(U_C) \ar[r,"\psi"]                        &      U_C    
	\end{tikzcd}
	\]
	Moreover, if $\underline{c}=(c_1,\ldots,c_d)\in \psi^{-1}(U_C)$, 
	then by \cite[Lemma 6.5]{GS} $\tilde{\psi}$ induces an isomorphism
	$$\prod \mb P(E_{c_i})=\pi_d^{-1}(\underline{c})\xrightarrow{\sim} \Phi^{-1}(\psi(\underline{c}))\,.$$
	Applying Lemma \ref{pullback-lemma} by taking $T=U_Z$, $S=\mc U$
	and $\phi=\tilde \psi$
	and the definition of 
	the map $\tilde \psi$ (see the proof of \cite[Theorem 2.2(a)]{G19}) we see that
	$$\tilde{\psi}^{*} \mc O_{\mc Q}(1)=\bigotimes\limits^d_{i=1}p^*_i\mc O_{\mb P(E)}(1)|_{U_Z}\,.$$
	Hence it is enough to show that $\exists~n\in \mb Z$ such that 
	$\forall~\underline{c}\in \psi^{-1}(U_C)$ 
	$$\tilde{\psi}^*\mc L|_{\pi_d^{-1}(\underline{c})}\cong \bigotimes_{i=1}^d p^*_i\mc O(n)|_{\pi_d^{-1}(\underline{c})}\,.$$   
	For $\underline{c}\in \psi^{-1}(U_C)$ define $n_i(\underline{c})\in \Z$ 
	using the equation 
	$$\tilde\psi^*\mathcal{L}|_{ \pi_d^{-1}(\underline{c})}=
			\bigotimes_{i=1}^d p_i^*\mathcal{O}_{\mb P(E_{c_i})}(n_i(\underline{c}))\,.$$ 
	We may 
	view the $n_i$ as functions $n_i:\psi^{-1}(U_C)\to \Z$.
	Since the line bundle $\t\psi^*\mc L$ 
	is invariant under the action of the group $S_d$, it follows that 
	\begin{equation}\label{permute-index-line-bundles}
		n_{\sigma(i)}(\underline{c})=n_i(\sigma(\underline{c})).
	\end{equation}
	Here $\sigma(\underline{c}):=(c_{\sigma(1)},\ldots,c_{\sigma(d)})$.
	Hence it suffices to show that $n_1$ is a constant function. 
	
	Let $c_2,\ldots,c_d$ be distinct points in $C$.
	Define $V:=C\setminus \{c_{2},..,c_{d}\}$ 
	and a map 
	$$i:V\hookrightarrow \psi^{-1}(U_{C})\qquad \qquad i(c):= (c,c_{2},..,c_{d})\,.$$ 
	Then $\pi^{-1}_d(V)$ is equal to 
	$\mathbb{P}(E|_{V})\times \mathbb{P}(E_{c_{2}})\times ... \times \mathbb{P}(E_{c_{d}})$.
	The restriction of $\tilde{\psi}^*\mc L$ to 
	$\mathbb{P}(E|_{V})\times \mathbb{P}(E_{c_{2}})\times ... \times \mathbb{P}(E_{c_{d}})$
	is isomorphic to 
	$$\pi^*M\otimes p_1^*\mc O_{\mb P(E\vert_V)}(a_1)\otimes p_2^*\mc O_{\mb P(E_{c_2})}(a_2)\ldots
	\otimes p_d^*\mc O_{\mb P(E_{c_d})}(a_d)\,,$$
	where $M$ is a line bundle on $V$. Further restricting to $(c,c_2,\ldots,c_d)$
	and $(c',c_2,\ldots,c_d)$, where $c,c'\in V$, we see that 
	\begin{equation}\label{Picard-computation-e1}
		n_i(c,c_{2},..,c_{d})=n_i(c',c_2,\ldots,c_d)\qquad \qquad \forall~i\,.
	\end{equation}
	This proves that for distinct points $c,c',c_2,\ldots,c_{d}\in C$ we have 
	\begin{equation}\label{Picard-computation-e2}
		n_i(c,c_{2},..,c_{d})=n_i(c',c_2,\ldots,c_d)\qquad \qquad \forall~i\,.
	\end{equation}
	Choose $2d$ distinct points $c_1,\ldots,c_d,c'_1,\ldots,c'_d$ in $C$. 
	Then using equations \eqref{Picard-computation-e1} and \eqref{Picard-computation-e2} we get
	\begin{align*}
	n_1(c_1,c_2,\ldots,c_d)&=n_1(c'_1,c_2,\ldots,c_d)\\
	&=n_2(c_2,c'_1,\ldots,c_d)\\
	&=n_2(c'_2,c'_1,c_3,\ldots,c_d)\\
	&=n_1(c'_1,c'_2,c_3,\ldots,c_d)\\
	&=\ldots\\
	&=n_1(c'_1,c'_2,\ldots,c'_d)\,.
	\end{align*}
	Finally, for any two points $\underline{c},\underline{c}'\in \psi^{-1}(U_C)$
	choose a third point $\underline{c}''$ such that the coordinates of $\underline{c}''$ 
	are distinct from those of $\underline{c}$ and $\underline{c}'$. Then we see that 
	$n_1(\underline{c})=n_1(\underline{c}'')=n_1(\underline{c}')$. This proves that $n_1$ 
	is the constant function.
	Therefore, $\psi^*\mathcal{L}|_{ \pi_d^{-1}(\underline{c})}$ is of the form 
	$\bigotimes p_i^*\mathcal{O}_{\mb P(E_{c_i})}(n)$, $\forall \underline{c}\in \psi^{-1}(U_C)$. The uniqueness of $n$ is obvious.
\end{proof}

\begin{theorem}\label{thm1}
${\rm Pic}(\mathcal{Q})=\Phi^*{\rm Pic}(C^{(d)})\bigoplus \Z[\mathcal{O}_{\mathcal{Q}}(1)] $\,.
\end{theorem}

\begin{proof}
Let $\mc L\in {\rm Pic}(\mc Q)$. 
By \cite[Corollary 6.3]{GS} and \cite[Corollary 6.4]{GS} 
the morphism $\Phi$ is flat and fibres of $\Phi$ are 
integral. Then by \cite[Lemma 2.1.2]{MR-ss} and 
Lemma \ref{generically the line bundle is O(n)} we get 
that $\mc L\otimes \mc O_{\mc Q}(-n)=\Phi^*\mc M$ for some 
$\mc M\in {\rm Pic}(C^{(d)})$. Hence 
$\mc L=\Phi^*\mc M\otimes \mc O_{\mc Q}(n)$. The uniqueness 
of such an expression follows from the statement on 
uniqueness in Lemma \ref{generically the line bundle is O(n)}.
\end{proof}

For a  projective variety $X$ over $k$ recall that  
$N^1(X)$ (respectively, $N_1(X)$) is the vector space of 
$\mb R$-divisors (respectively, $1$-cycles) modulo 
numerical equivalences \cite[\S 1.4]{Laz}. It is known 
that $N^1(X)$ and $N_1(X)$ are finite dimensional and 
the intersection product defines a non-degenerate pairing
$$N^1(X)\times N_1(X)\to \mb R\qquad \qquad ([\beta],[\gamma])\mapsto [\beta]\cdot [\gamma]\,.$$

We will compute $N^1(\mc Q)$ and $N_1(\mc Q)$. 
Let $\underline{c}\in U_C\subset C^{(d)}$. As we saw 
in the proof of Theorem \ref{thm1},  
$$\Phi^{-1}(\underline{c})\cong \prod \mb P(E_{c_i})\,.$$ 
Let $\mb P^1\hookrightarrow \mb P(E_{c_1})$ be a line 
and let $v_i\in \mb P(E_{c_i})$ for $i\geq 2$. Then we have an embedding:
\begin{equation}\label{line-fiber-Phi}
	\mb P^1\cong \mb P^1\times v_2\times\ldots\times v_d\hookrightarrow \mb P(E_{c_1}) \times \prod\limits_{i\geq 2}\mb P(E|_{c_i})=\Phi^{-1}(\underline{c})\subset \mc Q\,.
\end{equation}
\begin{definition}\label{line-fiber-hc}
Let us denote the class of this curve in $N_1(\mc Q)$ by $[l]$. 
\end{definition}

\begin{corollary} \label{neron-severi of quot scheme}
 $N^1(\mc Q)=\Phi^*N^1(C^{(d)})\bigoplus \R[\mathcal{O}_{\mathcal{Q}}(1)]$.
\end{corollary}

\begin{proof}
Since $\Phi$ is surjective, 
$N^1(C^{(d)})\to N^1(\mc Q)$ is an inclusion \cite[Example 1.4.4]{Laz}. 
Note that $\mc O_{\mc Q}(1)\neq 0$ in $N^1(\mc Q)$
since $[\mc O_{\mc Q}(1)]\cdot [l]=1$. 
Hence $\mc O_{\mc Q}(1)\neq 0$ in $N^1(\mc Q)$. 
This also shows that $\mc O_{\mc Q}(1)\notin \Phi^*N^1(C^{(d)})$.
 
By theorem \ref{thm1}, 
we know that any $N^1(\mc Q)$ is generated by 
$\Phi^*N^1(C^{(d)})$ and $[\mc O_{\mc Q}(1)]$. 
The only thing left is to show that 
$$\Phi^*N^1(C^{(d)}) \cap \mb R[\mc O_{\mc Q}(1)]=0\,.$$
For $a\in \mb R$ if $a[\mc O_{\mc Q}(1)]\in N^1(C^{(d)})$, 
then $a[\mc O_{\mc Q}(1)]\cdot [l]=a=0$. Hence the result follows.
\end{proof}

Hence, it follows from Corollary \ref{neron-severi of quot scheme} that 

\begin{proposition} \label{dim-N^1}
	If $g=1$ or $C$ is very general with $g\geq 2$, then
	${\rm dim}_\R N^1(\mc Q)=3$. 
\end{proposition}
\begin{proof}
	We already saw that $N^1(C^{(d)})$ is of dimension $2$.
	The Proposition follows.
\end{proof}

To compute $N_1(\mc Q)$ 
we first construct a section of $\Phi:\mc Q\to C^{(d)}$. 
Over $C\times C^{(d)}$ we have the universal divisor $\Sigma$ 
which gives us the universal quotient 
$\mc O_{C\times C^{(d)}}\to \mc O_{\Sigma}$. 
Choose a surjection 
$E\to L$ over $C$, where $L$ is a line bundle on $C$. 
This induces a surjection 
$E\otimes \mc O_{C\times C^{(d)}}\to L\otimes \mc O_{C\times C^{(d)}}$.
 Then the composition 
$$E\otimes \mc O_{C\times C^{(d)}} \to L\otimes \mc O_{C\times C^{(d)}} \to L\otimes \mc O_{\Sigma}$$ 
gives us a morphism 
\begin{equation}\label{section-HC}
	\eta:C^{(d)}\to \mc Q
\end{equation}
which is easily seen to be a section of $\Phi$.
 
\begin{corollary}
$N_1(\mc Q)=N_1(C^{(d)})\oplus \mb R[l]$
where $N_1(C^{(d)})\hookrightarrow N_1(\mc Q)$ 
is the morphism given by the pushforward $\eta_*$.
\end{corollary}

\begin{proof}
Since $\Phi\circ \eta=id_{C^{(d)}}$ we have that $\eta_*$ is an injection. 
Also since $[\mc O_{\mc Q}(1)]\cdot [l]=1$, we have $[l]\neq 0$. 
We claim that $[l]\notin N_1(C^{(d)})$. If not, assume that 
$[l]=\eta_*[\gamma]$ for $[\gamma]\in N^1(C^{(d)})$. Then for every $\beta\in N^1(C^{(d)})$
we have 
$$[l]\cdot \Phi^*\beta=\Phi_*([l])\cdot \beta = 0=\gamma \cdot \beta\,.$$
This proves that $\gamma=0$. 

Let $\gamma \in N_1(\mc Q)$. Then we claim that 
$$\gamma=\eta_*\Phi_*\gamma+\bigg([\mc O_{\mc Q}(1)]\cdot (\gamma-\eta_*\Phi_*\gamma)\bigg)[l]\,\,.$$
This can be seen as follows. It is enough to show that $\forall~D\in N^1(\mc Q)$, 
$$[D]\cdot \gamma=[D]\cdot (\eta_*\Phi_*\gamma)+([\mc O_{\mc Q}(1)]\cdot \gamma)[D]\cdot [l]\,.$$
By Corollary \ref{neron-severi of quot scheme}, 
it is enough to consider the case when $D=\Phi^*D'$ 
where $D'\in N^1(C^{(d)})$ or $D=\mc O_{\mc Q}(1)$. 
In the first case the statement follows from projection formula 
and the second case is by definition. This completes the 
proof of the Corollary.
\end{proof}

Let $p_C:C\times \mc Q \to \mc Q$ and $p_{\mc Q}:C\times \mc Q \to C$ 
be the projections. Let $\mc B_{\mc Q}$ denote the universal quotient on 
$C\times \mc Q$. 
For a vector bundle $F$ over $C$, we 
define 
$$B_{F,\mc Q}:={\rm det}(p_{\mc Q*}(\mc B_{\mc Q}\otimes p^*_CF))\,.$$

\begin{lemma}\label{lb-lemma}
	Suppose we are given a map $f:T\to \mc Q$. Let 
	$(id\times f)^*\mc B_{\mc Q}=\mc B_T$. Let $p_T:C\times T\to T$ and 
	$p_{1,T}:C\times T \to C$ be the projections. 
	\[
	\begin{tikzcd}
	C\times T \ar[r,"id\times f"] \ar[d,"p_T"] & C\times \mc Q \ar[d,"p_{\mc Q}"] \\
	T \ar[r,"f"] & \mc Q
	\end{tikzcd} 
	\]
	\begin{enumerate}[(i)]
		\item $f^*p_{\mc Q*}(\mc B_{\mc Q}\otimes p^*_C F)\to p_{T*}(\mc B_T\otimes p^*_{1,T}F)$ is an isomorphism.
		\item For a vector bundle $F$ on $C$ define 
		$B_{F,T}:= {\rm det}(p_{T*}(\mc B_T\otimes p^*_{1,T}F))$. Then $f^*B_{F,\mc Q}=B_{F,T}$.
	\end{enumerate}
\end{lemma}

\begin{proof}
	For (i) take $\mc B_{\mc Q}\otimes p^*_CF$ and use Lemma \ref{pullback-lemma}. 
	The assertion (ii) follows from (i) by applying determinant to the isomorphism 
	$$f^*p_{\mc Q*}(\mc B_{\mc Q}\otimes p^*_C F)\xrightarrow{\sim} p_{T*}(\mc B_T\otimes p^*_{1,T}F)\,.$$
\end{proof}

Recall the definition of $\eta$ from equation \eqref{section-HC}, this is a section of $\Phi$. 
For a line bundle $L$ on $C$ we have a line bundle $\mc G_{d,L}$ over $C^{(d)}$ 
(see \cite[page 8]{Pa} for notation).

\begin{lemma}\label{restriction of O(1) to symmetric product}
	Let $\eta$ be defined by a quotient $E\to M\to 0$. 
	Then $$\eta^*B_{L,\mc Q}\cong \mc G_{d,L\otimes M}\,.$$
\end{lemma}

\begin{proof}
	We have the diagram:
	\[
	\begin{tikzcd}
	C\times C^{(d)} \ar[r,"id_C\times \eta"] \ar[d] & C\times \mc Q \ar[d] \\
	C^{(d)} \ar[r,"\eta"]                     &         \mc Q
	\end{tikzcd}
	\]
	Recall that by definition of $\eta$, the pullback of 
	the universal quotient on $C\times \mc Q$
	to $C\times C^{(d)}$ via the section $(id_C\times \eta)$ is the quotient
	$$E\otimes \mc O_{C\times C^{(d)}} \to L\otimes \mc O_{C\times C^{(d)}} \to L\otimes \mc O_{\Sigma}$$
	Hence by Lemma \ref{lb-lemma}, we have 
	$$\eta^*B_{L,\mc Q}\cong \mc G_{d,L\otimes M}\,.$$
\end{proof}

\begin{proposition}\label{difference-B_L}
	For any two line bundles $L,L'$ over $C$
	$$B_{L,\mc Q}\otimes B^{-1}_{L',\mc Q}=\Phi^*((L\otimes L'^{-1})^{\boxtimes d})\,.$$
\end{proposition}
\begin{proof}
	First we show that $B_{L,\mc Q}\otimes B^{-1}_{L',\mc Q}\in \Phi^*{\rm Pic}(C^{(d)})$.
	Since any line bundle over $\mc Q$ is of the form 
	$\mc O_{\mc Q}(a)\otimes \phi^*\mc L$, where $\mc L\in {\rm Pic}(C^{(d)})$, 
	it is enough to show that both $B_{L,\mc Q}$ and $B_{L',\mc Q}$ have the 
	same $\mc O_{\mc Q}(1)$-th coeffcient. 
	
	To compute the coefficient of this component of any line 
	bundle over $\mc Q$, we can do the following.
	Fix $d$ distinct points $c_1,\ldots,c_d\in C$.
	These define a point $\underline{c}\in C^{(d)}$. As we saw 
	in the proof of Theorem \ref{thm1},  
	$$\Phi^{-1}(\underline{c})\cong \prod_{i=1}^d \mb P(E_{c_i})\,.$$ 
	Let $v_i\in \mb P(E_{c_i})$ for $i\geq 2$. Then we have an embedding:
	$$f:\mb P(E_{c_1})\times v_2\times\ldots\times v_d\hookrightarrow \mb P(E_{c_1}) \times \prod\limits_{i\geq 2}\mb P(E_{c_i})=\Phi^{-1}(\underline{c})\,.$$
	Then the $\mc O_{\mc Q}(1)$-th coefficient of a line 
	bundle $\mc M$ over $\mc Q$ is the degree of $f^*\mc M$ 
	with respect to $\mc O_{\mb P(E_{c_1})}(1)$. 
	Let $Y=\mb P(E_{c_1})$. 
	Using Lemma \ref{lb-lemma}, $f^* B_{L,\mc Q}={\rm det}(p_{Y*}(\mc B_Y\otimes p^*_{1,Y}L))$. 
	
	The $v_j\in \mb P(E_{c_j})$ correspond to 
	quotients $v_j:E\to E_{c_j}\to k_{c_j}$, 
	for $2\leq j \leq d$. Over $C\times Y$ 
	we have the inclusions $i_j:Y\cong c_j\times Y \hookrightarrow C\times Y$
	for every $1\leq j\leq d$.
	We have a map
	$$p_{1,Y}^*E \to \bigoplus\limits_{j=1}^{d} i_{j*}(p_{1,Y}^*E\vert_{c_j\times Y})\,.$$
	The bundle $p_{1,Y}^*E\vert_{c_j\times Y}$ is just the trivial bundle 
	on $Y$, and using $v_j$ we can get quotients $p_{1,Y}^*E\vert_{c_j\times Y}\to \mc O_Y$
	for $2\leq j\leq d$. For $j=1$ we have the quotient $p_{1,Y}^*E\vert_{c_1\times Y}\to i_{1*}(\mc O_Y(1))$.
	Since the $c_j\times Y$ are disjoint we can put these together to 
	get a quotient on $C\times Y$
	$$p_{1,Y}^*E  \to \left( \bigoplus\limits_{j=2}^{d} i_{j*}\mc O_{Y}\right) 
	\bigoplus i_{1*}\mc O_Y(1)\,.$$
	By definition, the sheaf $\mc B_Y$ is the sheaf in the RHS. 
	Then 
	\begin{align*}
	\mc B_Y\otimes p^*_{1,Y}L= & \left( \bigoplus\limits_{j=2}^{d} i_{j*}\mc O_{Y}\right) \otimes p^*_{1,Y}L
	\,\,\,  \bigoplus \,\,\, i_{1*}\mc O_Y(1)\otimes p^*_{1,Y}L \\
	= & \left( \bigoplus\limits_{j=2}^{d} i_{j*}\mc O_{Y}\right) \bigoplus i_{1*}\mc O_Y(1)\\
	= & \mc B_Y\,\,.
	\end{align*}
	Thus, using the remark in the preceding para, we get that the 
	$\mc O_{\mc Q}(1)$-th coefficient of $B_{L,\mc Q}$ is 
	the same as that of $B_{L',\mc Q}$. Hence $B_{L,\mc Q}\otimes B^{-1}_{L',\mc Q}=\Phi^*\mc L$.
	
	Recall the section $\eta$ 
	of $\Phi$ from equation \eqref{section-HC}, constructed using 
	some line bundle quotient $E\to M$.
	Then $\eta^*(B_{L,\mc Q}\otimes B^{-1}_{L',\mc Q})=s^*\Phi^*\mc L=\mc L$. 
	Now using Lemma \ref{restriction of O(1) to symmetric product}, 
	we get that $\eta^*B_{L,\mc Q}=\mc G_{d,L\otimes M}$. 

	By G\"ottsche's theorem (\cite[page 9]{Pa}) we get that 
	$\eta^*B_{L,\mc Q}=\mc G_{d,L\otimes M}=(L\otimes M)^{\boxtimes d}\otimes \mc O(-\Delta_d/2)$. 
	Therefore, we get 
	$$\mc L=\eta^*(B_{L,\mc Q}\otimes B^{-1}_{L',\mc Q})=
	(L\otimes L'^{-1})^{\boxtimes d}\,.$$
	This completes the proof of the Proposition.
\end{proof}

\begin{corollary}\label{compare-B_L-O(1)}
	$[B_{L,\mc Q}]=[\mc O_{\mc Q}(1)]+{\rm deg}(L)[x]$ in $N^1(\mc Q)$.
\end{corollary}

\section{Upper bound on NEF cone}

Let $V$ be a vector space of dimension $n$. 
From now, unless mentioned otherwise, 
the notation $\mc Q$ will be reserved for the space 
$\mc Q(V\otimes \mc O_C,d)$. Sometimes we will also 
denote this space by $\mc Q(n,d)$ when we want to emphasize  
$n$ and $d$. \\
{\bf Notation}. For the rest of this article, except in section \ref{genus-0}, 
the genus of the curve $C$ will be $g(C)\geq 1$.
If $g(C)\geq 2$ then we will also assume that $C$ is very general.

Our aim is to compute the NEF cone of $\mc Q$.
Since this cone is dual to the cone of effective
curves, it follows that if we take effective curves
$C_1,C_2,\ldots,C_r$, take the cone generated by these 
in $N_1(\mc Q)$, and take the dual cone $T$ in $N^1(\mc Q)$,
then ${\rm Nef}(\mc Q)$ is contained in $T$. This gives us an upper 
bound on ${\rm Nef}(\mc Q)$.
We already know two curves in $\mc Q$. The first 
being a line in the fiber of $\Phi:\mc Q\to C^{(d)}$,
see Definition \ref{line-fiber-hc}, which was denoted $[l]$. 
Recall the section $\eta$ 
of $\Phi$ from equation \eqref{section-HC}, taking $L$ to be the trivial
bundle.
The second curve is $\eta_*([l'])$, where $[l']$ is from Definition 
\ref{p1-in-C^(d)}. Now we will construct a third curve in $\mc Q$.

Define a morphism 
\begin{equation}\label{t delta}
	\tilde{\delta}:C\to \mc Q
\end{equation}
as follows. Let $p_1,p_2:C\times C\to C$ be the first and second projections respectively. 
Let $i:C\to C\times C$ be the diagonal. 
Fix a surjection $k^n\to k^d$ of vector spaces. 
Then define the quotient over $C\times C$
$$ \mc O^n_{C\times C} \to  \mc O^d_{C\times C} \to i_*i^* \mc O^d_{C\times C}\,.$$ 
This induces a  morphism $\tilde{\delta}:C\to \mc Q$ 
which sends $c \mapsto [\mc O^n_C \to k^d_c\to 0]$.
We will abuse notation and denote the class  
$[\tilde \delta_*(C)]\in N_1(\mc Q)$ by $[\tilde \delta]$.

We now give an upper bound for the NEF cone when $n\geq d\geq {\rm gon}(C)$.

\begin{proposition} \label{nef cone of quot schemes } 
	Consider the Quot scheme $\mc Q=\mc Q(n,d)$.
	Assume $n\geq d \geq {\rm gon}(C)$.
	Let $\mu_0:=\dfrac{d+g-1}{dg}$. Then 
	$${\rm Nef}(\mc Q)\subset \mb R_{\geq 0}([\mc O_{\mc Q}(1)]+\mu_0[L_0])+\mb R_{\geq 0}[\theta_d]+\mb R_{\geq 0}[L_0]\,.$$
\end{proposition}

\begin{proof}
	We claim that the cone dual to $\langle [l],\eta_*([l']),[\tilde \delta]\rangle$
	is precisely 
	$$\langle ([\mc O_{\mc Q}(1)]+\mu_0[L_0]),[L_0],[\theta_d]\rangle\,.$$
	We have the following equalities:
	\begin{enumerate}
		\item $([\mc O_{\mc Q}(1)]+\mu_0[L_0])\cdot [l]=1$. This is clear.
		\item $([\mc O_{\mc Q}(1)]+\mu_0[L_0])\cdot \eta_*[l']=0$. By 
		projection formula and Lemma \ref{restriction of O(1) to symmetric product}, 
		we get that 
		$$([\mc O_{\mc Q}(1)]+\mu_0[L_0])\cdot [\eta_*l']=([-\Delta_d/2]+\mu_0[L_0])\cdot [l']\,.$$
		By Lemma \ref{E in terms of various bases} we get that 
		$[-\Delta_d/2]+\mu_0[L_0]=(1-\mu_0)[\theta_d]$. But as we saw earlier, 
		$[\theta_d]\cdot [l']=0$.
		\item $([\mc O_{\mc Q}(1)]+\mu_0[L_0])\cdot [\tilde{\delta}]=0$. By 
		Lemma \ref{pullback-lemma}, it is easy to see that 
		$[\mc O_{\mc Q}(1)]\cdot [\tilde{\delta}]=0$. By 
		projection formula, we get 
		$$([\mc O_{\mc Q}(1)]+\mu_0[L_0])\cdot [\tilde{\delta}]=
		[\mu_0L_0]\cdot [\Phi_*\tilde{\delta}]=[\mu_0L_0]\cdot [\delta]=0\,.$$
		\item $[\theta_d]\cdot [l]=[L_0]\cdot [l]=0$ follows using the projection formula. 	
	\end{enumerate}
	Now the claim follows from Proposition \ref{nef-cone-C^(d)}.
	As explained before, since ${\rm Nef}(\mc Q)$ is contained in the 
	dual to the cone $\langle [l],\eta_*([l']),[\tilde \delta]\rangle$,
	the proposition follows.
\end{proof}

When the genus $g=1$, we have the following improvement of 
Proposition \ref{nef cone of quot schemes }. 

\begin{proposition} \label{ub-g=1} 
	Let $C$ be a smooth projective curve of genus $g=1$.
	Consider the Quot scheme $\mc Q=\mc Q(n,d)$.
	Assume $d \geq {\rm gon}(C)=2$. Then 
	$${\rm Nef}(\mc Q)\subset \mb R_{\geq 0}([\mc O_{\mc Q}(1)]+[L_0])+\mb R_{\geq 0}[\theta_d]+\mb R_{\geq 0}[L_0]\,.$$
\end{proposition}

\begin{proof}
	We claim that the cone dual to $\langle [l],\eta_*([l']),\eta_*[\delta]\rangle$
	is precisely 
	$$\langle ([\mc O_{\mc Q}(1)]+[L_0]),[L_0],[\theta_d]\rangle\,.$$
	Let us check that $[([\mc O_{\mc Q}(1)]+[L_0])]\cdot \eta_*[\delta]=0$.
	Since $[L_0]\cdot [\delta]=0$ it is clear that it suffices to check that 
	$[\mc O_{\mc Q}(1)]\cdot \eta_*[\delta]=0$. Applying the definition of 
	the map $\eta\circ\delta:C\to \mc Q$
	we see that $[\mc O_{\mc Q}(1)]\cdot \eta_*[\delta]={\rm deg}(p_{2*}(\mc O/\mc I^d))$,
	where $\mc I$ is the ideal sheaf of the diagonal in $E\times E$.
	Since $\mc I/\mc I^2$ is trivial and $\mc I^j/\mc I^{j+1}=(\mc I/\mc I^2)^{\otimes j}$,
	it follows that ${\rm deg}(p_{2*}(\mc O/\mc I^d))=0$. 
	The rest of the proof is the same as that of Proposition \ref{nef cone of quot schemes }.
\end{proof}

\section{Lower bound on NEF cone}

In this section we obtain a lower bound for ${\rm Nef}(\mc Q)$ ($\mc Q=\mc Q(n,d)$).

\begin{lemma}\label{B_L.D geq 0}
    Let $f:D\to \mc Q$ be a morphism, where $D$ is a smooth projective curve. Fix a point 
    $q\in f(D)$ and an effective divisor $A$ on $C$ containing the scheme theoretic support of     
    $\mc B_q$. If there is a line bundle $L$ on $C$ such that $H^0(L)\to H^0(L|_A)$ is surjective
    then $[B_{L,\mc Q}]\cdot [D]
    \geq 0$. 
\end{lemma}

\begin{proof}
    Consider the map 
	$$p_{\mc Q*}(p_C^*(V\otimes \mc O_C)\otimes p_C^*L)\to p_{\mc Q*}(\mc B_{\mc Q}\otimes p_C^*L)$$
	on $\mc Q$. We claim that this map is surjective at the point $q$. 
	In view of Lemma \ref{pullback-lemma} when 
	we restrict this map to $q$, it becomes equal to the map 
	$$H^0(V\otimes L)\to H^0(\mc B_q\otimes L)\,.$$
	The map $V\otimes L\to \mc B_q\otimes L$ on $C$ factors as 
	$$V\otimes L\to V\otimes L\vert_A \to \mc B_q\otimes L\,.$$
	Taking global sections we see that the map $H^0(V\otimes L)\to H^0(\mc B_q\otimes L)$
	factors as 
	$$H^0(V \otimes L)\to H^0(V\otimes L\vert_A) \to H^0(\mc B_q\otimes L)\,.$$
	The second arrow is surjective since these are coherent sheaves on a zero 
	dimensional scheme. The first arrow is simply  
	$$V\otimes H^0(L)\to V\otimes H^0(L\vert_A)\,.$$
	Since $H^0(L)\to H^0(L\vert_A)$ is surjective by our choice 
	of $L$, it follows that  $H^0(V\otimes L)\to H^0(\mc B_q\otimes L)$ is surjective,
	and so it follows that 
	$p_{\mc Q*}(V\otimes p_C^*L)\to p_{\mc Q*}(\mc B_{\mc Q}\otimes p_C^*L)$
	is surjective at the point $q$.
	
	The rank of the vector bundle $p_{\mc Q*}(\mc B_{\mc Q}\otimes p_{C}^*L)$ on $\mc Q$ is $d$. Taking the $d$th 
	exterior of $p_{\mc Q*}(V\otimes p_C^*L)\to p_{\mc Q*}(\mc B_{\mc Q}\otimes p_C^*L)$ 
	we get a map 
	$$\bigwedge^d(V\otimes H^0(L))\to B_{L,\mc Q}\,.$$
	This map is nonzero and that can be seen by looking at the restriction to the 
	point $q$. 
	This shows that there is a global section of $B_{L,\mc Q}$ whose restriction to $q$ does 
	not vanish. It follows that $[B_{L,\mc Q}]\cdot [D]\geq 0$.	This completes
	the proof of the lemma.
\end{proof}

\begin{lemma}\label{line bundles restricted to a divisor}
	Let $A$ be an effective divisor on $C$ of degree $d$.
	Then  there is a line bundle
	$L$ of degree $d+g-1$ such that the natural map 
	$$H^0(L)\to H^0(L\vert_A)$$
	is surjective.
\end{lemma}

\begin{proof}
	It suffices to find 
	a line bundle of degree $d+g-1$ such that 
	$H^1(L\otimes \mc O_C(-A))=0$. By Serre duality this is 
	same as saying that $H^0(L^\vee\otimes K_C\otimes \mc O_C(A))=0$. The 
	degree of $L^\vee\otimes K_C\otimes \mc O_C(A)$ is $g-1$. Thus, fixing $A$
	we may choose a general $L$ such that $L^\vee\otimes K_C\otimes \mc O_C(A)$ 
	line bundle has no global sections. 
\end{proof}

\begin{definition}\label{def-U}
Define $U\subset \mc Q$ to be the set of 
quotients of the form 
$$	\mc O^n_{C}\to \dfrac{\mc O_{C}}{\prod_{i=1}^r\mf m^{d_i}_{C,c_i}} \cong 
		\bigoplus \dfrac{\mc O_{C,c_i}}{\mf m^{d_i}_{C,c_i}}\qquad\qquad c_i \neq c_j\,. $$
\end{definition}

We now prove a lemma, 
which is implicitly contained \cite[Section 5]{GS}.
Let $\Sigma\subset C\times C^{(d)}$ denote the closed sub-scheme which 
is the universal divisor. 
In the following Lemma we work more generally with $\mc Q(E,d)$.
\begin{lemma}
	Let $E$ be a locally free sheaf of rank $r$ on $C$. 
	Let $\mc Q=\mc Q(E,d)$ denote the Quot scheme of torsion 
	quotients of length $d$.
	The universal quotient 
	$\mc B_{\mc Q}$ is supported on $\Phi^*\Sigma\subset C\times \mc Q$. 
	The set $U$ is open in $\mc Q$. 
	On $C\times U$ the sheaf $\mc B_{\mc Q}$ is a line bundle supported 
	on the scheme $\Phi^*\Sigma\cap (C\times U)$.
\end{lemma}

\begin{proof}
	Let $A$ denote the kernel of the universal quotient on 
	$C\times \mc Q$
	$$0\to A\xrightarrow{h} p_C^*E\to \mc B_{\mc Q}\to 0\,.$$
	The map $\Phi$ is defined taking the determinant of 
	$h$, that is, using the quotient 
	$$0\to {\rm det}(A)\xrightarrow{{\rm det}(h)} p_C^*{\rm det}(E) \to \mc F\to 0\,.$$
	If $\mc I_\Sigma$ denotes the ideal sheaf of $\Sigma$
	then this shows that 
	$$\Phi^*\mc I_{\Sigma}={\rm det}(A)\otimes p_C^*{\rm det}(E)^{-1}\,.$$
	Let $0\to E'\xrightarrow{h} E$ be locally free sheaves of the same rank on 
	a scheme $Y$. Let $\mc I$ denote the ideal sheaf determined by 
	${\rm det}(h)$. Then it is easy to see that $\mc IE\subset h(E')\subset E$.
	Applying this we get that $(\Phi^*\mc I_\Sigma) p_C^*E \subset A$.
	This proves that $\mc B$ is supported on $\Phi^*\Sigma$.
	Let us denote by $Z:=\Phi^*\Sigma\subset C\times \mc Q$.
	Consider the closed subset $Z_2\subset Z$ defined as follows
	$$Z_2:=\{z=(c,q)\in Z\,\,\vert\,\, {\rm rank}_k(\mc B_{\mc Q}\otimes k(z))\geq 2 \}\,.$$
	Then the image of $Z_2$ in $\mc Q$ is closed and $U$ is precisely
	the complement of $Z_2$. This proves that $U$ is open in $\mc Q$.

	Let $R$ be a local ring with maximal ideal $\mf m$ and let $R\to S$ be a finite map.
	Let $M$ be a finite $S$ module, which is flat over $R$ and
	such that $M/\mf mM\cong S/\mf mS$. Then it follows easily that $M\cong S$. 
	
	Let $q\in U\subset \mc Q$ be a point. The sheaf 
	$\mc B_{\mc Q}$ is a coherent sheaf supported on $Z$, the map $Z\to \mc Q$
	is finite, the fiber 
	$$\mc B_q=\bigoplus \dfrac{\mc O_{C,c_i}}{\mf m^{d_i}_{C,c_i}} \cong \mc O_{\Sigma}\vert_q\cong \mc O_Z\vert_q\,. $$ 
	From the preceding remark it follows that $\mc B_{\mc Q}$ is a line bundle 
	over $Z\cap (C\times U)$.
\end{proof}

\begin{lemma} \label{the open set U}
	Consider the Quot scheme $\mc Q=\mc Q(n,d)$.
	Let $D$ be a smooth projective curve and let 
	$D\to \mc Q$ be a morphism such that its image intersects $U$. Then 
	$([\mc O_{\mc Q}(1)]+[\Delta_d/2])\cdot [D]\geq 0$.
\end{lemma} 
\begin{proof}
	Denote by $\mc B_D$ the pullback of the universal 
	quotient over $C\times \mc Q$ to $C\times D$. 
	Denote by $i_D:\Gamma\hookrightarrow C\times D$ the pullback of the universal 
	subscheme $\Sigma\hookrightarrow C\times C^{(d)}$ to $C\times D$. Then 
	$\mc B_D$ is supported on $\Gamma$.
	
	Let $\Gamma_i$ be the irreducible components of $\Gamma$. 
	Since $\Gamma \to D$ is flat each $\Gamma_i$ dominates $D$. 
	Let $f:\Gamma \to D$ denote the projection. 
	There is an open subset $U_1\subset D$ such that 
	$$f^{-1}(U_1)=\bigsqcup_i \Gamma_i\cap f^{-1}(U_1)$$
	and $\mc B_D$ restricted to $f^{-1}(U_1)$ is a line bundle. 
	Note that by $\Gamma_i\cap f^{-1}(U_1)$ we mean this open sub-scheme
	of $\Gamma$.
	Fix a closed point $x_i\in \Gamma_i\cap f^{-1}(U_1)$. Consider the quotient 
	$$V\otimes \mc O_{C\times D}\to \mc B_D$$
	and restrict it to the point $x_i$. We get a quotient 
	$$V\to \mc B_D\otimes k(x_i)\to 0\,.$$
	If we pick a general line in $V$, then it surjects onto $\mc B_D\otimes k(x_i)$.
	Thus, for the general element $s\in V$, $s\otimes \mc O_{C\times D}$ 
	surjects onto $\mc B_D\otimes k(x_i)$.
	This map factors through $\mc O_\Gamma$, and we get an 
	exact sequence
	$$0 \to \mc O_{\Gamma} \to \mc B_D \to F \to 0$$
	where $F$ is supported on a 0 dimensional scheme. 
	Then we have
	$$0 \to f_*\mc O_{\Gamma} \to f_*\mc B_D \to f_*F \to 0\,.$$
	Since $f_*F$ is again supported on finitely many points, hence we have
	$${\rm deg}(f_*\mc B_D)-{\rm deg}(f_*\mc O_{\Gamma})\geq 0$$
	By Lemma \ref{pullback-lemma}, ${\rm deg}(f_*\mc B_D)=[\mc O_{\mc Q}(1)]\cdot [D]$ and by \cite[\S 3]{Pa} we have 
	$${\rm deg}(f_*\mc O_{\Gamma})=[\mc O(-\Delta_d/2)]\cdot [D]\,.$$ 
	Hence the result follows.
\end{proof}

\begin{corollary}\label{cor-image-meets-U}
	If the image of $f:D\to \mc Q$ interects $U$, then
	$([\mc O_{\mc Q}(1)]+\mu_0[L_0])\cdot [D]\geq 0$.
\end{corollary}
\begin{proof}
If its image interects $U$, then by Lemma \ref{the open set U}, 
$$([\mc O_{\mc Q}(1)]+[\Delta_d/2])\cdot [D]\geq 0\,.$$ 
By Lemma \ref{E in terms of various bases}, 
$$[\Delta_d/2]=\mu_0[L_0]-(1-\mu_0)[\theta_d]\,.$$
Since $\theta_d$ is nef, we have that
$$([\mc O_{\mc Q}(1)]+\mu_0[L_0])\cdot [D]\geq 0\,.$$
\end{proof}

\begin{lemma} \label{the complement of open set U}
	Consider the Quot scheme $\mc Q=\mc Q(n,d)$.
	Let $D$ be a smooth projective curve and let 
	$f:D\to (\mc Q\setminus U) \subset \mc Q$ be a morphism. 
	Then $([\mc O_{\mc Q}(1)]+(d+g-2)[x])\cdot [D]\geq 0$.
\end{lemma} 

\begin{proof}
	Fix a point $q\in f(D)$. 
	Let $A$ be the scheme theoretic support of the quotient $\mc B_{q}$ on 
	$C$. Let ${\rm deg}(A)=d'$. Since $q\notin \mc U$, we have $d'<d$. By
	Lemma \ref{line bundles restricted to a divisor} we have a line bundle 
	$L$ of degree $d'+g-1$ such that $H^0(L)\to H^0(L|_A)$ is surjective.
	By Lemma \ref{B_L.D geq 0} and Corollary \ref{compare-B_L-O(1)} we  
	get that $[B_{L,\mc Q}]\cdot [D]=([\mc O_{\mc Q}(1)]+(d'+g-1)[x])\cdot [D]\geq 0$. Since $ 
	[x]$ is nef on $\mc Q$	and $d'\leq d-1$ we get that $([\mc O_{\mc Q}(1)]+(d+g-2)
	[x])\cdot [D]\geq  
	0$.	
\end{proof}

\begin{proposition}\label{some line bundle is nef }
	Consider the Quot scheme $\mc Q=\mc Q(n,d)$.
	Let $\mu_0=\dfrac{d+g-1}{dg}$. 
	Then the class 
	$\kappa_1:=[\mc O_{\mc Q}(1)]+\mu_0[L_0]+\dfrac{d+g-2}{dg}[\theta_d]$ is nef. 
\end{proposition}

\begin{proof}
	Let $D\to \mc Q$ is a morphism, where $D$ is a smooth projective curve. 
	If the image of this morphism intersects $U$ then by Lemma \ref{the open set U}
	we have $([\mc O_{\mc Q}(1)]+ [\Delta_d/2])\cdot [D]\geq 0$. 
	By Lemma \ref{E in terms of various bases} we have
	$[\Delta_d/2]=\mu_0[L_0]-(1-\mu_0)[\theta_d]$. 
	Hence we get 
	$$([\mc O_{\mc Q}(1)]+\mu_0[L_0])\cdot [D]\geq (1-\mu_0)[\theta_d]\cdot [D]\geq 0\,.$$
	Since $[\theta_d]$ is nef, we get 
	$$([\mc O_{\mc Q}(1)]+\mu_0[L_0]) \cdot [D] +\dfrac{d+g-2}{dg}[\theta_d] \cdot [D] 
		\geq 0\,. $$
	Now assume $D\to \mc Q$ does not intersect $U$. Then by Lemma 
	\ref{the complement of open set U} we get 
	$$([\mc O_{\mc Q}(1)] +(d+g-2)[x])\cdot [D]\geq 0\,.$$
	By Lemma \ref{E in terms of various bases} we have 
	$[x]=\dfrac{1}{dg}[L_0]+\dfrac{1}{dg}[\theta_d]$.
	Therefore 
	\begin{align*}
		(d+g-2)[x] = & \dfrac{d+g-2}{dg}[L_0]+\dfrac{d+g-2}{dg}[\theta_d]     \\
		= & \mu_0[L_0]-\dfrac{1}{dg}[L_0]+\dfrac{d+g-2}{dg}[\theta_d]\,.
	\end{align*}
	Since $L_0$ is nef we get that 
	$$([\mc O_{\mc Q}(1)]  +\mu_0[L_0]  +\dfrac{d+g-2}{dg}[\theta_d]) \cdot [D] \geq 0 \,. $$
\end{proof}

\begin{lemma}\label{B_L-not-ample}
	Let $L$ be a line bundle on $C$ of degree $d+g-1$.
	If $d\geq {\rm gon}(C)$ then the line bundle $B_{L,\mc Q}$
	is not ample. Moreover, for any $t\in [0,1]$ the 
	class $t[B_{L,\mc Q}]+(1-t)[\theta_d]$ is nef but not ample.
\end{lemma}
\begin{proof}
	We saw in the last para of the proof of Proposition \ref{difference-B_L}
	that $\eta^*B_{L,\mc Q}=L^{\boxtimes d}\otimes \mc O(-\Delta_d/2)$.
	Its class in the nef cone is $(d+g-1)[x]-[\Delta_d/2]$. It follows from
	Lemma \ref{E in terms of various bases} that this is equal to $[\theta_d]$. 
	Since $d\geq {\rm gon}(C)$ we have $\theta_d$ is not ample on $C^{(d)}$. 
	That $t[B_{L,\mc Q}]+(1-t)[\theta_d]$ is nef is clear since 
	both $[B_{L,\mc Q}]$ and $[\theta_d]$ are nef. This is not ample
	since $\eta^*$ of this class is $[\theta_d]$ on 
	$C^{(d)}$, which is not ample. 
\end{proof}

\begin{proposition}\label{nef-B_L}
	Consider the Quot scheme $\mc Q=\mc Q(n,d)$.
	Then the class $[\mc O_{\mc Q}(1)]+(d+g-1)[x]\in N^1(\mc Q)$
	is nef.
\end{proposition}

\begin{proof}
	It is easily checked that the class $[\mc O_{\mc Q}(1)]+(d+g-1)[x]$
	can be written as a positive linear combination of $[\theta_d]$
	and the class in Proposition \ref{some line bundle is nef }.
\end{proof}

We may slightly improve Proposition \ref{nef-B_L} in a special case using the 
results in \cite{Pa}. For this we first recall the main
results in \cite[\S4]{Pa}. Let $C$ be a very general curve of genus $g(C)=2k$.
Since the gonality is given by $\lfloor\frac{g+3}{2}\rfloor$, 
in this case it is $k+1$. Let $L_i'$ denote the finitely
many $g^1_{k+1}$'s on $C$ and define $L_i=K_C-L_i'$.  
Then ${\rm deg}(L_i)=3(k-1)$.
It is proved in \cite[Proposition 3.6, Theorem 4.1]{Pa} that $\mc G_{k,L_i}$ 
is nef but not ample. 

\begin{proposition}\label{nef-B_L-Pa}
	Let $C$ be a very general curve of genus $g(C)=2k$. 
	Consider the Quot scheme $\mc Q=\mc Q(n,k)$.
	The line bundle $B_{L,\mc Q}$ is nef when 
	${\rm deg}(L)\geq 3(k-1)$. When ${\rm deg}(L)=3(k-1)$ 
	the class $t[B_{L,\mc Q}]+(1-t)[\mc G_{k,L}]$ is nef but not ample
	for any $t\in [0,1]$.
\end{proposition}

\noindent 
We remark that this is an improvement
since Proposition \ref{nef-B_L} only shows that $B_{L,\mc Q}$
is nef when ${\rm deg}(L)\geq 3k-1$. 

\begin{proof}
	It follows from Proposition \ref{difference-B_L} that the class
	of $B_{L,\mc Q}$ in $N^1(\mc Q)$ is $[\mc O_{\mc Q}(1)]+{\rm deg}(L)[x]$,
	since $B_{\mc O_C,\mc Q}=\mc O_{\mc Q}(1)$. Notice that this class only 
	depends on the degree of $L$. Since the sum of nef line
	bundles is nef, it suffices to show that $[B_{L,\mc Q}]=[\mc O_{\mc Q}(1)]+{\rm deg}(L)[x]$ 
	is nef when ${\rm deg}(L)=3(k-1)$.
	
	The set $V(\sigma_{L_i})$ is defined in equation
	\cite[equation (18)]{Pa}. Then (A) in \cite[Theorem 4.1]{Pa}
	says that for every $A\in C^{(k)}$ there is an $L_i$ 
	such that $H^0(C,L_i)\to H^0(C,L_i\vert_A)$ is surjective. 
	
	Let $f:D\to \mc Q$ be morphism, where $D$ is a smooth 
	projective curve. Fix a point $q\in f(D)$.
	Let $A$ be the divisor corresponding to
	$\Phi(q)$, then $A$ is an effective divisor of degree $k$. 
	For this $A$, choose a line bundle $L_i$ such that 
	$$H^0(C,L_i)\to H^0(C,L_i\vert_A)$$ 
	is surjective. The scheme theoretic support 
	of $\mc B_q$ is contained in $A$. It follows from Lemma \ref{B_L.D geq 0}
	that 
	$$f^*B_{L_i,\mc Q}=f^*([\mc O_{\mc Q}(1)]+3(k-1)[x])\geq 0\,.$$	
	It follows that $B_{L,\mc Q}$ is nef.
	
	Note that 
	\begin{align*}
	\eta^*B_{L,\mc Q}&=\eta^*[\mc O_{\mc Q}(1)]+{\rm deg}(L)\eta^*[x]\\
	&=[\mc O(-\Delta_k/2)]+3(k-1)[x]\\
	&=[\mc G_{k,L}]\,.
	\end{align*}
	Thus, when $t\in [0,1]$ the pullback along $\eta$ 
	of $t[B_{L,\mc Q}]+(1-t)[\mc G_{k,L}]$ is $[\mc G_{k,L}]$, which 
	is not ample.
\end{proof}

\section{The genus 0 case}\label{genus-0}

Throughout this section we will work with $C=\mb P^1$.
Let us first compute the nef cone of $\mc Q(n,d)$.

Note that we have $C^{(d)}\cong \mb P^d$. 
Hence $N^1(C^{(d)})=\mb R[\mc O_{\mb P^d}(1)]$. 
By Corollary \ref{neron-severi of quot scheme} it follows
that $N^1(\mc Q)$ is two dimensional. Hence, it suffices
to find a line bundle on $\mc Q$ which is different from 
the pullback of $\mc O_{\mb P^d}(1)$ and which is nef but 
not ample. 
The following result is proved in 
\cite[Theorem 6.2]{Str}, but we include it for the 
benefit of the reader.

\begin{proposition}
\begin{align*}    	
	{\rm Nef}(\mc Q(n,d)) = & \mb R_{\geq 0}[B_{\mc O(d-1),\mc Q}]+\mb R_{\geq 0}[\mc O_{\mb P^d}(1)]                                                \\
	                      = & \mb R_{\geq 0}([\mc O_{\mc Q}(1)]+(d-1)[\mc O_{\mb P^d}(1)])+\mb R_{\geq 0}[\mc O_{\mb P^d}(1)]\,.
\end{align*}	
\end{proposition}
\begin{proof}
	Let $W:=H^0(\mb P^1,\mc O_{\mb P^1}(d))$.
	There is a natural isomorphism $\mb PW^*\xrightarrow{\sim} C^{(d)}$.
	The universal sub-scheme $\Sigma\subset \mb P^1\times \mb PW^*$
	is given by the tautological section 
	$$p_2^*\mc O_{\mb PW^*}(-1)\to p_2^*W=p_1^*W\to p_1^*\mc O_{\mb P^1}(d)\,.$$
	
	By Lemma \ref{B_L.D geq 0} and Lemma \ref{line bundles restricted to a divisor} 
	we get that $B_{\mc O(d-1),\mc Q}$ is nef. To show $B_{\mc O(d-1),\mc Q}$ is not ample, 
	consider a section $\eta:C^{(d)}\to \mc Q$ constructed as in 
	(\ref{section-HC}) with $L$ the trivial bundle. 
	Let $p_i$ denote the two projections from $\mb P^1\times \mb PW^*$.
	By definition and Lemma \ref{lb-lemma} it follows that 
	$\eta^*B_{\mc O(d-1),\mc Q}={\rm det}(p_{2*}(\mc O_\Sigma\otimes p_1^*\mc O_{\mb P^1}(d-1)))$.
	Tensoring the exact sequence 
	$$0\to p_1^*\mc O_{\mb P^1}(-d)\otimes p_2^*\mc O_{\mb PW^*}(-1)\to \mc O_{\mb P^1\times \mb PW^*} \to \mc O_{\Sigma}\to 0$$
	with $p_1^*\mc O_{\mb P^1}(d-1)$ and applying $p_{2*}$ 
	it easily follows that $p_{2*}(\mc O_\Sigma\otimes p_1^*\mc O_{\mb P^1}(d-1))$
	is the trivial bundle and so $\eta^*B_{\mc O(d-1),\mc Q}$ is trivial. 
	This proves that $B_{\mc O(d-1),\mc Q}$ is nef but not ample. 
	
	By restricting to a fiber of $\Phi$ and using Corollary \ref{compare-B_L-O(1)}
	we see that $[B_{\mc O(d-1),\mc Q}]$ is linearly independent from
	$[\mc O_{\mb P^d}(1)]$. 
	This completes the proof of the first equality. The second equality will follow 
	from the first equality once we show that 
	$$[B_{\mc O(d-1),\mc Q}]=[\mc O_{\mc Q}(1)]+(d-1)[\mc O_{\mc P^d}(1)]\,.$$ 
	By Corollary \ref{compare-B_L-O(1)}, we have that 
	$[B_{\mc O(d-1),\mc Q}]=[\mc O_{\mc Q}(1)]+(d-1)[x]$. 
	Now recall that given $x\in \mb P^1$, 
	$[x]$ is the class of the divisor in $C^{(d)}$ whose underlying set consists of 
	effective divisors of degree $d$ containing $x$ (see (\ref{three divisors of 
	symmetric product of curves})). Hence,  $[x]$ is the class of the  
	hyperplane section 
	$$\mb P(H^0(\mb P^1,\mc O(d)\otimes \mc O(-x)))^*) \subset \mb 
	P(H^0(\mb P^1,\mc O(d))^*)=C^{(d)}\,.$$
	 Therefore $[x]=[\mc O_{\mb P^1}(1)]$ and this completes the proof of the second 
	 equality.
\end{proof}

\begin{theorem}\label{nef cone of quot schemes over projective line}
	Let $C=\mb P^1$. Let $E=\bigoplus\limits^k_{i=1} \mc O(a_i)$ with $a_i\leq a_j$ for $i<j$. 
	Let $d\geq 1$. Let $L=\mc O(-a_1+d-1)$.
	Then 
\begin{align*}	
	{\rm Nef}(\mc Q(E,d)) = &  \mb R_{\geq 0}[B_{L,\mc Q(E,d)}]+\mb R_{\geq 0}[\mc O_{\mb P^d}(1)]                                                 \\
                       	  = &  \mb R_{\geq 0}([\mc O_{\mc Q(E,d)}(1)]+(-a_1+d-1)[\mc O_{\mb P^d}(1)])+\mb R_{\geq 0}[\mc O_{\mb P^d}(1)]\,.
\end{align*}
\end{theorem}

\begin{proof}
	By Corollary \ref{neron-severi of quot scheme} we get that  
	$N^1(\mc Q(E,d))$ is $2$-dimensional. 
	Hence it is enough to give two line bundles which are nef but not ample. 
	Clearly $\Phi^*_{\mc Q(E,d)}\mc O_{\mb P^d}(1)$ is nef but not ample.
	So it is enough to show 
	that $B_{L,\mc Q(E,d)}$ is nef but not ample.
	
	Since $a_j-a_1\geq 0~\forall~j\geq 1$, we get that $E(-a_1)$ is globally generated. Let 
	$V:=H^0(C,E(-a_1))$ and let ${\rm dim}~V=n$. Then we have a surjection 
	$V\otimes \mc O_C \to E(-a_1)$. Then gives us a surjection
	$$V\otimes \mc O_C\to p^*_CE(-a_1)\to \mc B_{\mc Q(E,d)}\otimes p^*_C\mc O_C(-a_1) \to 0\,.$$
	This defines a map $f:\mc Q(E,d)\to \mc Q(n,d)$. By Lemma \ref{lb-lemma} we get that 
	$$f^*B_{\mc O(d-1),\mc Q(n,d)}=B_{L,\mc Q(E,d)}={\rm det}(p_{\mc Q(E,d)*}(\mc B_{\mc Q(E,d)}\otimes p_C^*L))\,.$$ 
	Since $B_{\mc O(d-1),\mc Q(n,d)}$ is nef we get that $B_{L,\mc Q(E,d)}$ is nef.
	We next show that the $B_{L,\mc Q(E,d)}$ is not ample. 
	Consider the section $\eta_{\mc Q(E,d)}$ of 
	$\Phi_{\mc Q(E,d)}:\mc Q(E,d)\to C^{(d)}$ defined by the quotient 
	$p_C^*E\to p_C^*\mc O(a_1)\otimes \mc O_{\Sigma}$ on $C\times C^{(d)}$
	(see (\ref{section-HC})).
	Then $f\circ \eta_{\mc Q(E,d)}$ is a section of $\Phi:\mc Q(n,d)\to C^{(d)}$ 
	defined by a quotient $\mc O^n_C\to \mc O_\Sigma\to 0$ on $C\times C^{(d)}$. Therefore 
	$\eta_{\mc Q(E,d)}^*B_{L,\mc Q(E,d)}=\eta^*B_{\mc O(d-1),\mc Q(n,d)}$.
	As $\eta^*B_{\mc O(d-1),\mc Q(n,d)}$ is not ample,  
	we get that $B_{L,\mc Q(E,d)}$ is not ample.
	The second equality follows again from the fact that $[x]=[\mc O_{\mb P^d}(1)]$.
\end{proof}

\section{Some cases of equality}
Now we are back to the assumption that 
the genus of the curve satisfies $g(C)\geq 1$
and if $g(C)\geq 2$ then we also assume that $C$ is very general. 
\begin{definition}\label{def-U'}
Let $U'\subset \mc Q$ be the open set consisting of 
quotients $\mc O^n_C\to B\to 0$ such that the induced 
map $H^0(C,\mc O^n_C)\to H^0(C,B)$ is surjective.
\end{definition}
\begin{lemma}\label{the open set V}
	Consider the Quot scheme $\mc Q=\mc Q(n,d)$.
	Let $D$ be a smooth projective curve and let 
	$D\to \mc Q$ be a morphism such that its image intersects $U'$. Then 
	$[\mc O_{\mc Q}(1)]\cdot [D]\geq 0$.
\end{lemma}
\begin{proof}
	We continue with the notations of Lemma \ref{the open set U}. 
	Let $p_D:C\times D\to D$ be the projection. Then applying 
	$(p_D)_*$ to the quotient $\mc O^n_{C\times \mc Q}\to \mc B_D$ we get that the morphism
	$$(p_D)_*\mc O^n_{C\times D}=\mc O^n_D\to (p_D)_*\mc B_D$$
	is generically surjective by our assumption and Lemma \ref{pullback-lemma}. Hence we get that 
	$$[\mc O_{\mc Q}(1)]\cdot [D]={\rm deg}((p_D)_*\mc B_D)\geq 0\,.$$ 
\end{proof}

One extremal ray in ${\rm Nef}(C^{(2)})$ is given by $L_0$.
Let other extremal ray of ${\rm Nef}(C^{(2)})$ be given by 
\begin{equation}\label{def-alpha_t}
	\alpha_t=(t+1)x-\Delta_2/2\,,
\end{equation} 
(see \cite[page 75]{Laz}). Then using Lemma \ref{E in terms of various bases}, we get that
\begin{equation}\label{relation between L_0, Delta, alpha}
	\Delta_2/2=\dfrac{t+1}{g+t}L_0-\dfrac{g-1}{g+t}\alpha_t\,.
\end{equation}

\begin{theorem}\label{cone-d=2}
	Let $d=2$. Consider the Quot scheme $\mc Q=\mc Q(n,2)$. 
	Then $${\rm Nef}(\mc Q)=\mb R_{\geq 0}([\mc O_{\mc Q}(1)]+\dfrac{t+1}{g+t}[L_0])+
		\mb R_{\geq 0}[L_0]+\mb R_{\geq 0}[\alpha_t]\,.$$ 
\end{theorem}
\begin{proof}
	We first prove that $[\mc O_{\mc Q}(1)]+\dfrac{t+1}{g+t}[L_0]$ is nef.
	Since $d=2$, then there are only three types of quotients:
	\begin{enumerate}
		\item $\mc O^n_C\to \dfrac{\mc O_{C,c_1}}{\mf m_{C,c_1}}\oplus 
		\dfrac{\mc O_{C,c_2}}{\mf m_{C,c_2}}$ with $c_1\neq c_2$\,,
		\item $\mc O^n_C\to \dfrac{\mc O_{C,c_1}}{\mf m^2_{C,c_1}}$\,,
		\item $\mc O^n_C\to \dfrac{\mc O_{C,c}}{\mf m_{C,c}}\oplus \dfrac{\mc O_{C,c}}{\mf m_{C,c}}$\,.
	\end{enumerate}
	The first two quotients are in $U$ while the third one is in $U'$, that is,
	we get $U \cup U' =\mc Q$. Now let $D$ be a smooth projective curve and $D\to \mc Q$ 
	be a morphism. If its image interects $U$, then by Corollary \ref{cor-image-meets-U}, 
	$([\mc O_{\mc Q}(1)]+\Delta_2/2)\cdot [D]\geq 0\,.$
	Using (\ref{relation between L_0, Delta, alpha}) and the fact that $\alpha_t$ is  
	nef, we get that $([\mc O_{\mc Q}(1)]+\dfrac{t+1}{g+t}[L_0])\cdot [D]\geq 0$.
	If $D$ does not intersect $U$ then $D\subset U'$. Hence by Lemma \ref{the open set 
		V}, 
	we have 
	$$[\mc O_{\mc Q}(1)]\cdot [D]\geq 0\,.$$
	Since $[L_0]$ is nef we have that
	$$([\mc O_{\mc Q}(1)]+\dfrac{t+1}{g+t}[L_0])\cdot [D]\geq 0\,.$$
	Also $([\mc O_{\mc Q}(1)]+\dfrac{t+1}{g+t}[L_0])\cdot [\t \delta]=0$. Hence any 
	convex linear combination of $[\mc O_{\mc Q}(1)]+\dfrac{t+1}{g+t}[L_0]$ and 
	$[L_0]$ is nef but not ample. By (\ref{relation between L_0, Delta, 
		alpha}) $
	\eta^*([\mc O_{\mc Q}(1)]+\dfrac{t+1}{g+t}[L_0])=\dfrac{g-1}{g+t}\alpha_t$. 
	Hence any convex linear combination of $[\mc O_{\mc Q}(1)]+\dfrac{t+1}{g+t}[L_0]$ 
	and $[\alpha_t]$ is not ample. 
	Hence the result follows.
\end{proof}
Precise values for $t$ depending on $g$ are known when 
\begin{enumerate}
	\item When $g=1$, $t=1$.
	\item When $g=2$, $t=2$.
	\item When $g=3$, $t=9/5$.
	\item When $g$ is a perfect square $t=\sqrt{g}$, see \cite[Theorem 2]{Kouv-1993}.
	\item In \cite[Propn. 3.2]{Ciro-Kouv}, when 
	$g\geq 9$, assuming the Nagata conjecture, they prove that $t=\sqrt{g}$. 
\end{enumerate}
Thus, in all these cases using Theorem \ref{cone-d=2} we get 
the Nef cone of $\mc Q(n,2)$.\\

\subsection{\bf Criterion for nefness.}\label{subsection-criterion-for-nefness}
In the remainder of this section, we will need to work with
$C^{(d)}$ for different values of $d$. The line bundles $L_0$ 
on $C^{(d)}$ will therefore be denoted by $L_0^{(d)}$ when 
we want to emphasize the $d$. Similarly, we will denote 
$\mu_0^{(d)}=\dfrac{d+g-1}{dg}$.
Let $\mc P^{\leq n}_{(d)}$ be the set of all partitions 
$(d_1,d_2,\ldots,d_k)$ of $d$ of length at most $n$. 
Given an element $\mathbf{d}\in \mc P^{\leq n}_{(d)}$ define 
$$C^{(\mathbf{d})}:=C^{(d_1)}\times C^{(d_2)}\times \ldots \times C^{(d_k)}$$ 
and if $p_{i}:C^{(\mathbf{d})}\to C^{(d_i)}$ is the $i$-th projection
we define a class
$$[\mc O(-\Delta_{\mathbf{d}}/2)]:=[\sum p^*_{i}\mc O(-\Delta_{d_i}/2)] \in
		N^1(C^{(\mathbf{d})})\,.$$  
Note that we have a natural addition 
$$\pi_{\mathbf{d}}:C^{(\mathbf{d})}\to C^{(d)}\,.$$
For a partition $\mathbf{d}\in \mc P^{\leq n}_d$ define a morphism 
$$\eta_{\mathbf{d}}:C^{(\mathbf{d})}\to \mc Q$$
as follows.
For any $l\geq 1$, we define the universal 
subscheme of $C^{(l)}$ over $C\times C^{(l)}$ by $\Sigma_l$. 
Then over $C\times C^{(\mathbf{d})}$ we have the subschemes 
$(id\times p_{i})^*\Sigma_{d_i}$. We have a quotient 
$$q_{\mathbf{d}}:\mc O^n_{C\times C^{(d)}}\to \bigoplus\limits_i \mc O_{(id\times p_{i,\mathbf{d}})^*\Sigma_{d_i}}$$
defined by taking direct sum of morphisms 
$\mc O_{C\times C^{(d)}}\to \mc O_{(id\times p_{i,\mathbf{d}})^*\Sigma_{d_i}}$. 
Then $q_{\mathbf{d}}$ defines a map $C^{(\mathbf{d})}\to \mc Q$. 
By Lemma \ref{lb-lemma}, we have
\begin{equation}\label{pullback of O(1) via various sections}
	[\eta^*_{\mathbf{d}}\mc O_{\mc Q}(1)]=[\mc O(-\Delta_{\mathbf{d}}/2)]\,.
\end{equation}

\begin{lemma}\label{bound on O(1) in terms of diagonals}
	Let $D$ be a smooth projective curve.
	Let $D\to \mc Q$ be a morphism. Then there exists a 
	partition $\mathbf{d}\in \mc P^{\leq n}_{(d)}$ such that the composition $D\to \mc Q\to C^{(d)}$ factors as 
	$D\to C^{(\mathbf{d})}\to C^{(d)}$ and $[\mc O_{\mc Q}(1)]\cdot [D]\geq [\mc O(-\Delta_{\mathbf{d}}/2)]\cdot [D]$. 
\end{lemma}
\begin{proof}
	We will proceed by induction on $d$. When $d=1$ the statement is obvious.
	
	Let us denote the pullback of the universal quotient on $C\times \mc Q$ to $C\times D$ by $\mc B_D$ and let $f:C\times D\to D$ be the natural projection. 	
	Consider a section such that the composite $\mc O_{C\times D}\to \mc O_{C\times D}^n\to \mc B_D$
	is non-zero and let $\mc F$ denote the cokernel of the composite map.
	We have a commutative diagram
	\begin{equation}\label{diagram-inductive-step}
		\xymatrix{
			0\ar[r] & \mc O_{C\times D}\ar[r]\ar@{->>}[d] & \mc O_{C\times D}^n \ar[r]\ar@{->>}[d] & 
			\mc O_{C\times D}^{n-1}\ar[r]\ar@{->>}[d] & 0\\
			0\ar[r] & \mc O_{\Gamma'}\ar[r] & \mc B_D\ar[r] & \mc F\ar[r] & 0
		}
	\end{equation}
	Let $T_0(\mc F)\subset \mc F$ denote the maximal subsheaf of dimension $0$, see
	\cite[Definition 1.1.4]{HL}. Define $\mc F':=\mc F/T_0(\mc F)$. Now, either $\mc 
	F'=0$ or $\mc F'$
	is torsion free over $D$, and hence, flat over $D$. 
	In the first case, it follows that $D$ meets the open set $U$
	in Lemma \ref{the open set U}. Then we take $\mathbf{d}=(d)$ 
	and the statement follows from Lemma \ref{the open set U}. 
	So we assume $\mc F'$ is flat over $D$ and let $d'$ be 
	the degree of $\mc F'|_{C\times x}$, for $x\in D$. So 
	$0<d'<d$. By (\ref{diagram-inductive-step})
	we have
	$${\rm deg}~f_*\mc B_D= {\rm deg}~f_*\mc O_{\Gamma'}+{\rm deg}~f_*\mc F\,.$$
	Since $T_0(\mc F)$ is supported on finitely many points, we have 
	${\rm deg}~\mc F\geq {\rm deg}~\mc 
	F'$. In other words, we have 
	\begin{equation}\label{inequality required for the inductive step}
		{\rm deg}~f_*\mc B_D\geq {\rm deg}~f_*\mc O_{\Gamma'}+f_*\mc F'	 \,.	
	\end{equation}
	Now $\Gamma'$ defines a morphism $D\to C^{(d-d')}$ and note that
	$${\rm deg}~f_*\mc O_{\Gamma'}=[\mc O(-\Delta_{d-d'}/2)]\cdot [D]\,.$$
	The quotient $\mc O^{n-1}_{C
		\times D}\to \mc F' \to 0$ defines a map $D\to \mc Q(n-1,d')$. By induction 
	hypothesis, we get that there exists a partition 
	$\mathbf{d}'\in \mc P^{\leq n-1}_{d'}$ such  
	that the composition $D\to \mc Q(n-1,d')\to C^{(d')}$ factors as
	$D\to C^{(\mathbf{d'})}\to C^{(d')}$ and 
	$$[\mc O_{\mc Q(n-1,d')}(1)]\cdot [D]\geq [\mc O(-\Delta_{\mathbf{d'}}/2)]\cdot [D]\,.$$
	Since ${\rm deg}~f_*\mc F'=[\mc O_{\mc Q(n-1,d')}(1)]\cdot [D]$ we have that 
	${\rm deg}~f_*\mc F'\geq [\mc O(-\Delta_{\mathbf{d'}}/2)]\cdot [D]$.
	From (\ref{inequality required for the inductive step}) we get that 
	$$[\mc O_{\mc Q}(1)]\cdot [D]\geq [\mc O(-\Delta_{d-d'}/2)]\cdot D + [\mc O(-\Delta_{\mathbf{d'}}/2)]\cdot [D]\,.$$
	Now we define $\mathbf{d}:=(d-d',\mathbf{d'})$ and 
	the statement follows from the above inequality. 
\end{proof}

\begin{theorem}\label{criterion-for-nefness}
	Let $\beta\in N^1(C^{(d)})$.
	Then the class $[\mc O_{\mc Q}(1)]+\beta \in N^1(\mc Q)$ is nef iff the class 
	$[\mc O(-\Delta_{\mathbf{d}}/2)]+\pi^*_{\mathbf{d}}\beta\in N^1(C^{(\mathbf{d})})$ 
	is nef for all $\mathbf{d}\in \mc P^{\leq n}_d$.
\end{theorem}
\begin{proof}
	From (\ref{pullback of O(1) via various sections}) it is clear that if 
	$[\mc O_{\mc Q}(1)]+\beta$ is nef, then 
	$\eta^*_{\mathbf{d}}([\mc O_{\mc Q}(1)]+\beta)=[\mc O(-\Delta_{\mathbf{d}}/2)]+\pi^*_{\mathbf{d}}\beta$ is nef.
	
	For the converse, we assume 
	$[\mc O(-\Delta_{\mathbf{d}}/2)]+\pi^*_{\mathbf{d}}\beta$ is 
	nef for all $\mathbf{d}\in \mc P^{\leq n}_d$. Let $D$ be 
	a smooth projective curve and $D\to \mc Q$ be a morphism.
	By Lemma \ref{bound on O(1) in terms of diagonals} we have that there exists $\mathbf{d}\in \mc P^{\leq n}_d$ such that 
	$D\to C^{(d)}$ factors as $D\to C^{(\mathbf{d})}\to C^{(d)}$ and 
	$$[\mc O_{\mc Q}(1)]\cdot [D]\geq [\mc O(-\Delta_{\mathbf{d}}/2)]\cdot[D]\,.$$
	Now by assumption we have that 
	$$[\mc O(-\Delta_{\mathbf{d}}/2)]\cdot[D]\geq -\beta\cdot [D]\,.$$
	Therefore we get 
	$$[\mc O_{\mc Q}(1)]\cdot [D]\geq -\beta\cdot [D]\,.$$
	Hence we get that the class $[\mc O_{\mc Q}(1)]+\beta$ is nef.
\end{proof}

\begin{lemma} \label{compare various L_0}
	Suppose we are given a map $D\to C^{(\mathbf{d})} \xrightarrow{\pi_{\mathbf{d}}} C^{(d)}$.
	Then we have
	$$[L^{(d)}_0]\cdot [D]\geq \sum_{i}[L^{(d_i)}_0]\cdot [D]\,.$$ 
\end{lemma}
\begin{proof}
	By $[L_0^{d_i}]\cdot [D]$ we mean the degree of the pullback 
	of $[L_0^{(d_i)}]$ along $D\to C^{(\mathbf{d})}\xrightarrow{p_i} C^{(d_i)}$.
	The lemma follows easily from the definition of $L_0^{(d)}$ 
	and is left to the reader. 
	\end{proof}

\begin{proposition}\label{another lower bound for NEF cone}
	Let $n\geq 1$, $g\geq 1$ and $\mc Q=\mc Q(n,d)$.  
	Then the class $\kappa_2:=[\mc O_{\mc Q}(1)]+\dfrac{g+1}{2g}[L_0^{(d)}] \in N^1(\mc Q)$ is nef.
	As a consequence we get that 
	$${\rm Nef}(\mc Q)\supset \mb R_{\geq 0}\kappa_1+
		\mb R_{\geq 0}\kappa_2+\mb R_{\geq 0}[\theta_d]+\mb R_{\geq 0}[L_0^{(d)}]\,.$$

\end{proposition}

\begin{proof}
	Recall $\mu_0^{(2)}=\dfrac{g+1}{2g}$.
	By Theorem \ref{criterion-for-nefness}
	it suffices to show that for all $\mathbf{d}\in \mc P^{\leq n}_{(d)}$
	we have $[\mc O(-\Delta_{\mathbf{d}}/2)]+\mu_0^{(2)}\pi^*_{\mathbf{d}}[L_0^{(d)}]$ is nef.
	Using Lemma \ref{E in terms of various bases}, 
	$[L^{(1)}_0]=0$ and Lemma \ref{compare various L_0} we get 
	\begin{align*}
		([\mc O(-\Delta_{\mathbf{d}}/2)]+\mu_0^{(2)}\pi^*_{\mathbf{d}}[L_0^{(d)}])\cdot [D]&=
			\Big(\sum_i(1-\mu_0^{(d_i)})[\theta_{d_i}]-\mu_0^{(d_i)}[L_0^{d_i}]\Big)\cdot [D] \\
			&\qquad \qquad +\mu_0^{(2)}[L_0^{(d)}]\cdot [D]\\
			&\geq \sum_i(\mu_0^{(2)} - \mu_0^{(d_i)})[L_0^{d_i}]\cdot [D]\,.
	\end{align*}
	This proves that $\kappa_2$ is nef.
	That $\kappa_1$ is nef is proved in Proposition \ref{some line bundle is nef }.
	This completes the proof of the theorem.	
\end{proof}

\begin{corollary}\label{O(1)+mu^2L_0-not-ample}
	Let $n\geq d$. Then the class $[\mc O_{\mc Q}(1)]+\mu^{(2)}_0[L^{(d)}_0]\in N^1(\mc Q)$ is nef but not ample.
\end{corollary}
\begin{proof}
	By Proposition \ref{another lower bound for NEF cone} we have that 
	$[\mc O_{\mc Q}(1)]+\mu^{(2)}_0[L^{(d)}_0]$ is nef. Now recall that
	when $n\geq d$ we have the curve $\t \delta\hookrightarrow \mc Q$ (\ref{t delta}).
	From the definition of $\t \delta$ and  Lemma \ref{lb-lemma} we have 
	$[\mc O_{\mc Q}(1)]\cdot [\t \delta]=0$. Also $\Phi_*\t \delta=\delta$. Hence 
	$[L^{(d)}_0]\cdot [\t \delta]=[L^{(d)}_0]\cdot [\delta]=0$. From this we get 
	$[\mc O_{\mc Q}(1)]+\mu^{(2)}_0[L^{(d)}_0]\cdot [\t \delta]=0$ and hence 
	$[\mc O_{\mc Q}(1)]+\mu^{(2)}_0L^{(d)}_0$ is not ample.
\end{proof}

As a corollary we get the following result. When $g=1$ note that $\mu_0^{(2)}=1$.

\begin{theorem}\label{O(1)+Delta/2 is nef for elliptic curves}
	Let $g=1$, $n\geq 1$ and $\mc Q=\mc Q(n,d)$.  
	Then the class $[\mc O_{\mc Q}(1)]+[\Delta_d/2]\in N^1(\mc Q)$ is nef.
	Moreover, 
	$${\rm Nef}(\mc Q)= \mb R_{\geq 0}([\mc O_{\mc Q}(1)]+[\Delta_d/2])+
	\mb R_{\geq 0}[\theta_d]+\mb R_{\geq 0}[\Delta_d/2]\,.$$
\end{theorem}

\section{Curves over the small diagonal}\label{curves-over-small-diagonal}

Throughout this section the genus of the curve $C$ will 
be $g(C)\geq 2$ and $C$ is a very general curve. Recall that $\Phi:\mc Q\to C^{(d)}$ is the 
Hilbert-Chow map.
\begin{proposition}\label{prop-curve-over-small-diag}
Let $f:D\to \mc Q(n,d)$ be such that $\Phi\circ f$ factors 
through the small diagonal. Then $[\mc O_{\mc Q}(1)]\cdot [D]\geq 0$.
\end{proposition}
\begin{proof}
Since $\Phi\circ f$ factors through the small diagonal, there is a map 
$g:D\to C$ such that if $\Gamma:=\Gamma_g$ denotes the graph of 
$g$ in $C\times D$, and $\mc O_{C\times D}^n\to \mc B_D$ is the quotient 
on $C\times D$, then $\mc B_D$ is supported on $\mc O_{C\times D}/\ms I(\Gamma)^d$.  
Denote $\mc I:=\ms I(\Gamma)$.
Then $\mc B_D/\mc I\mc B_D$ is a globally generated sheaf on $D$
and so its determinant has degree $\geq 0$. Now consider the sheaf
$$\mc I^i\mc B_D/\mc I^{i+1}\mc B_D\cong (\mc I/\mc I^2)^{\otimes i}\otimes \mc B_D/\mc I\mc B_D\,.$$
Using adjunction it is easily seen that $\mc I/\mc I^2\cong g^*\omega_C$. 
Since ${\rm det}(\mc B_D/\mc I\mc B_D)$ has degree $\geq0$, it follows that 
${\rm det}(\mc I^i\mc B_D/\mc I^{i+1}\mc B_D)$ has degree $\geq 0$.
From the filtration 
$$\mc B_D\supset \mc I\mc B_D\supset \mc I^2\mc B_D\supset \ldots \supset \mc I^d\mc B_D=0$$
we easily conclude that $[\mc O_{\mc Q}(1)]\cdot [D]\geq 0$.
\end{proof}

\begin{lemma}
	Let $D\to C^{(d)}$ be a morphism. Then we can find a cover 
	$\t D\to D$ such that the composite $\t D\to D  \to C^{(d)}$ 
	factors through $C^d$. 
\end{lemma}
\begin{proof}
	Let $D_1$ be a component of $D\times_{C^{(d)}}C^d$ 
	which dominates $D$. Take $\t D$ to be a resolution of $D_1$. 
\end{proof}

\begin{corollary}\label{cover-D}
	Let $D\to \mc Q$ be a morphism. Replacing $D$ by a cover $\t D$
	we may assume that the map $\t D\to D \to \mc Q\to C^{(d)}$
	factors through $C^d$. 
\end{corollary}

In view of the above, given a map $D\to Q$ we may assume that the composite 
$D\to \mc Q\to C^{(d)}$ factors through $C^d$. Let each component be given
by a map $f_i:D\to C$. 
Denote by $i_D:\Gamma\hookrightarrow C\times D$ the pullback of the universal 
subscheme $\Sigma\hookrightarrow C\times C^{(d)}$ to $C\times D$. 
The ideal sheaf of $\Gamma$ is the product $\ms I(\Gamma_{f_i})$, the ideal sheaves of the 
graphs $\Gamma_{f_i}\subset C\times D$.
Moreover, $\mc B_D$ is supported on $\Gamma$.
Let $g_1,g_2,\ldots,g_r$ be the distinct maps in the set 
$\{f_1,f_2,\ldots,f_d \}$ and assume that $g_i$ occurs $d_i$ 
many times. Then we have $\ms I(\Gamma)=\prod_{i=1}^r\ms I(\Gamma_{g_i})^{d_i}$.
There is a natural map
$$\psi:\mc B_D\to \bigoplus \mc B_D/\ms I(\Gamma_{g_i})^{d_i}\mc B_D\,.$$
\begin{lemma}
	Let $f:D\to \mc Q$ be such that $\Phi\circ f$ factors through
	$C^d\to C^{(d)}$. If $\psi$ is an isomorphism then $[\mc O_{\mc Q}(1)]\cdot [D]\geq 0$.
\end{lemma}
\begin{proof}
	Since $\mc B_D$ is a quotient of $\mc O^n_{C\times D}$ it follows
	that each $\mc B_D/\ms I(\Gamma_{g_i})^{d_i}\mc B_D$ is a quotient 
	of $\mc O^n_{C\times D}$. Thus, each $\mc B_D/\ms I(\Gamma_{g_i})^{d_i}\mc B_D$
	defines a map $D \to \mc Q(n,d'_i)$ such that the image under 
	the map $\Phi:\mc Q(n,d'_i)\to C^{(d'_i)}$ is the small diagonal.
	By Proposition \ref{prop-curve-over-small-diag}
	it follows that degree of ${\rm det}(p_{D*}(\mc B_D/\ms I(\Gamma_{g_i})^{d_i}\mc B_D))$
	is $\ge 0$. Since $\psi$ is an isomorphism it follows that 
	degree of ${\rm det}(p_{D*}(\mc B_D))$ is $\geq 0$.
\end{proof}

We can use the above method to prove a result similar to Theorem 
\ref{cone-d=2} when $d=3$.

\begin{corollary}\label{nefness of the line bundle-d=3}
	Let $d=3$. Consider the Quot scheme $\mc Q=\mc Q(n,3)$.
	Let $\mu_0^{(3)}=\dfrac{g+2}{3g}$. Then $[\mc O_{\mc Q}(1)]+\mu_0^{(3)}[L_0^{(3)}]$ is nef. 
\end{corollary}
\begin{proof}
	If $d=3$ there are only these types of quotients:
	\begin{enumerate}
		\item $\mc O^n_C\to \mc O_{C}/\mf m_{C,c_1}\mf m_{C,c_2}\mf m_{C,c_3}$\,,\\
		\item $\mc O^n_C\to \mc O_{C,c_1}/{\mf m}_{C,c_1}\oplus 
			\mc O_{C}/\mf m_{C,c_1}\mf m_{C,c_2}$ \,,\\
		\item $\mc O^n_C\to \dfrac{\mc O_{C,c}}{\mf m_{C,c}}\oplus \dfrac{\mc O_{C,c}}{\mf m_{C,c}}
				\oplus \dfrac{\mc O_{C,c}}{\mf m_{C,c}}$\,.
	\end{enumerate}	
	Let $f:D\to \mc Q$ be a map. If $D$ contains a quotient 
	of type (1) or (3) then $D$ meets $U$ or $U'$(see Definition \ref{def-U} and 
	Definition \ref{def-U'}). Thus, in these cases 
	$([\mc O_{\mc Q}(1)]+\mu_0^{(3)}[L_0^{(3)}])\cdot [D]\geq 0$
	by Corollary \ref{cor-image-meets-U} and Lemma \ref{the open set V}. 
	
	Now consider the case when all points in the image of $D$ are of type (2).
	After replacing $D$ by a cover, using Corollary \ref{cover-D}, 
	we may assume that the map $D\to \mc Q$ factors through $C^3$.
	Since the images of points of $D$ represent quotients of 
	type (2), we may assume that the map
	from $D\to C^3$ looks like $d\mapsto (g_1(d),g_1(d),g_2(d))$.
	Now consider a general section $\mc O_{C\times D}\to \mc B_D$.
	Arguing as in the proof of Lemma \ref{the open set U}
	we get a diagram as in equation \eqref{diagram-inductive-step},
	such that $\mc O_{\Gamma'}$ defines a map $D\to C^{(2)}$ 
	and $\mc F'=\mc F/T_0(\mc F)$ is a line bundle on $D$ which
	is globally generated.
	Hence
	\begin{align*}	
	[\mc O_{\mc Q}(1)]\cdot [D]& \geq  [\mc O(-\Delta_{2}/2)]\cdot [D] + [c_1(p_{D*}(\mc F))]\cdot [D]\\
	&\geq  -\mu^{(2)}_0[L^{(2)}_0]\cdot [D]  \,.
	\end{align*}
	One easily checks using the definition of $L_0$ that in this case 
	$[L_0^{(3)}]\cdot [D]=2[L_0^{(2)}]\cdot [D]$. Thus,
	$$([\mc O_{\mc Q}(1)]+\mu_0^{(3)}[L_0^{(3)}])\cdot [D]\geq (2\mu_0^{(3)}-\mu_0^{(2)})[L_0^{(2)}]\cdot [D]\geq 0\,.$$
	This completes the proof of the Corollary.
\end{proof}

Combining this with Proposition \ref{nef cone of quot schemes } 
we get the following result. 

\begin{theorem}\label{cone-d=3}
	Let $C$ be a very general curve of genus $2\leq g(C)\leq 4$. 
	Let $n\geq3$ and let $\mc Q=\mc Q(n,3)$. Let $\mu_0=\dfrac{g+2}{3g}$
	Then 
	$${\rm Nef}(\mc Q)= \mb R_{\geq 0}([\mc O_{\mc Q}(1)]+\mu_0[L_0^{(3)}])+
		\mb R_{\geq 0}[\theta_d]+\mb R_{\geq 0}[L_0^{(3)}]\,.$$
\end{theorem}

\newcommand{\etalchar}[1]{$^{#1}$}



\end{document}